\documentclass[11pt]{article}
\usepackage{graphicx}
\usepackage{latexsym,amsmath,amsfonts}
\usepackage{times}
\usepackage{xcolor}
\usepackage[misc,geometry]{ifsym} 
\usepackage{lipsum}

\newcommand\blfootnote[1]{%
  \begingroup
  \renewcommand\thefootnote{}\footnote{#1}%
  \addtocounter{footnote}{-1}%
  \endgroup
}

\newtheorem{theorem}{Theorem}[section]

\newtheorem{definition}[theorem]{Definition}
\newtheorem{example}[theorem]{Example}

\newtheorem{lemma}[theorem]{Lemma}

\newtheorem{proposition}[theorem]{Proposition}
\newtheorem{remark}[theorem]{Remark}

\usepackage{lipsum} 
\newenvironment{proof}[1][Proof]{\noindent\textbf{#1.} }{\ \rule{0.5em}{0.5em}}
\begin{document}

\title{On a weighted Trudinger-Moser inequality in $\mathbb{R}^N$}

\author{Emerson Abreu\thanks{Email: eabreu@ufmg.br}\ \ \&\ Leandro G. Fernandes Jr\thanks{Email: lgfernandes@ufmg.br}\\
Departamento de Matem\'atica\\
Universidade Federal de Minas Gerais\\
\bigskip
CP702 - 30123-970 Belo Horizonte-MG, Brazil }
        \date{\vspace{-5ex}}
\maketitle

\begin{abstract}
	
We establish the Trudinger-Moser inequality on weighted Sobolev spaces in the whole space, and for a class of quasilinear elliptic operators in radial form of the type $\displaystyle Lu:=-r^{-\theta}
(r^{\alpha}\vert u'(r)\vert^{\beta}u'(r))',$ where $\theta, \beta\geq 0$ and $\alpha>0$, are constants satisfying some existence conditions. It worth emphasizing that these operators generalize the $p$-
Laplacian and $k$-Hessian operators in the radial case. Our results involve fractional dimensions, a new weighted P\'olya-Szeg{\"o} principle, and a boundness value for the optimal constant in a 
Gagliardo-Nirenberg type inequality.
\end{abstract}

\section{Introduction}
\label{intro}

\blfootnote{This study was financed by the Coordena\c c\~ao de Aperfei\c coamento de Pessoal de N\'{\i}vel Superior - Brasil (CAPES) and Funda\c c\~ao de Apoio \`a Pesquisa do Estado de Minas Gerais - (FAPEMIG)}
It is well known the classsical Sobolev embedding it holds that the embedding \linebreak$W^{1,p}(\Omega)\hookrightarrow L^{q}(\Omega)$ is continuous for any  \begin{math}p\leq q\leq Np/(N-p)
\end{math}, where \begin{math} p<N \end{math} and \begin{math}\Omega\end{math} a domain contained in \begin{math}\mathbb{R}^{N}\end{math}. Although, the embedding \begin{math}W^{1,N}
(\Omega)\hookrightarrow L^{q}(\Omega)\end{math} is continuous for any \begin{math}N\leq q<\infty\end{math}, \begin{math}W^{1,N}(\Omega)\not\subset L^{\infty}(\Omega)\end{math}.  Motivated by this
 approach Adams \cite{Adams} proved that for every \begin{math}0<\mu\leq 1\end{math} the Sobolev space \begin{math}W^{1,N}(\Omega)(\Omega\,  \mbox{unbounded})\end{math} is embedding in the
  Orlicz space \begin{math}L_{\Psi_{\mu,N}}(\Omega)\end{math}, where \begin{equation*}\label{0.1.0}
\Psi_{\mu,N}(t)=e^{\mu t^{\frac{N}{N-1}}}-\sum_{j=0}^{N-2}\frac{\mu^{j}}{j!}t^{\frac{N}{N-1}j}.
\end{equation*}

Hempel, Morris and Trudinger \cite{Trudingersharp} showed that the best Orlicz space $L_{\Psi}(\Omega)$ for the embedding of $W^{1,N}_{0}(\Omega)$ $($where $\Omega$ is a bounded domain in 
$\mathbb{R}^{N}$ $)$ occurs when $\Psi=\phi:=e^{t^{\frac{N}{N-1}}}-1$. More precisely, the space $W^{1,N}_{0}(\Omega)$ may not be contiuously imbedding in any Orlicz space $L_{\Psi}(\Omega)$
 whose defining function $\Psi$ increases strictly more rapidly than the function $\phi$.

The case when $\Omega$ is a bounded domain was studied by J. Moser in \cite{Moser}, which showed the following sharp result 
\begin{equation}\label{0.1.0.0}
\sup_{u\in W^{1,N}_{0}(\Omega)\backslash\{0\}} \frac{1}{|\Omega|}\int_{\Omega}e^{\mu\left(\frac{\vert u \vert}{\left\Vert\nabla u\right\Vert_{L^{p}}}\right)^{\frac{N}{N-1}}}dx \left\{
\begin{array}{cc}
\leq C(N,\mu), &\mbox{if}\ \mu\leq\mu_{N}\\
&\\
=+\infty, &\mbox{if}\ \mu>\mu_{N},  \\
\end{array}
\right.
\end{equation}
where $\mu_N:=N\omega_{N-1}^{\frac{1}{N-1}}$, $|\Omega|$ is a measure of $\Omega$, $\omega_{N-1}$ is the $(N-1)$-dimensional Lebesque measure of the unit sphere in $\mathbb{R}^{N}$, and
 $C(N,\mu)$ is a positive constant depending only on $N$ and $\mu$.

The case $\Omega=\mathbb{R}^N$, was studied by Ruf in \cite{Ruf} for $N=2$, and Li and Ruf in \cite{Li-Ruf} for $N\geq 3$. In all cases a sharp result as obtained. Namely, there exists $D(N,\mu)$
 which depends only on $N$ and $\mu$ satisfying
\begin{equation}\label{0.1.1}
\int_{\mathbb{R}^{N}}\Psi_{\mu,N}(u)dx\leq D(N,\mu)
\end{equation}
for all $u\in W^{1,N}(\mathbb{R}^{N})$ with $\left\Vert u\right\Vert_{W^{1,N}(\mathbb{R}^{N})}=1$ and $\mu\leq \mu_{N}$. Here, the inequality (\ref{0.1.1}) is not valid if $\mu>\mu_{N}$.

Ishiwata in \cite{Ishiwata} studied the attainability of the best constant 
\begin{equation}\label{0.1.1.0}
d_{N,\mu}:=\sup_{u\in W^{1,N}(\mathbb{R}^{N}):\left\Vert u\right\Vert_{W^{1,N}(\mathbb{R}^{N})}=1}\int_{\mathbb{R}^{N}}\Psi_{N,\mu}(u)dx,
\end{equation}
which is associated with (\ref{0.1.1}) [see section \ref{sec:2}, Theorem \ref{P.0.1.0.1} and Theorem \ref{P.0.1.0.2}]. A similar study was done in \cite{Ishiwata-Nakamura-Wadade} for singular weights.

He used a concentration-compactness type argument, proving that the maximizing sequence for (\ref{0.1.1}) are neither vanishing nor concentrating sequence. He also showed that the functional
 $J(u):=\int_{\mathbb{R}^{2}}\Psi_{2,\mu}(u)dx$ does not have critical points on $M:=\{u\in W^{1,2}(\mathbb{R}^{2}):\left\Vert u \right\Vert_{W^{1,2}(\mathbb{R}^{N})}=1 \}$ for $\mu$ sufficiently small,
  which implies non-existence results in this case.

Our approach for Trundiger-Moser inequality will be done for the class of quasilinear elliptic operators in radial form of the type
\begin{equation*}
Lu:=-r^{-\theta}(r^{\alpha}\vert u'(r)\vert^{\beta}u'(r))',
\end{equation*}
where $\theta, \beta\geq 0$ and $\alpha>0$. See \cite{Clement-Djairo-Enzo,Djairo-Miyagaki} for some problems involving the operator $L$. It worth emphasizing that these operators generalize the $p$
-Laplacian and $k$-Hessian operators in the radial case, more precisely,

\begin{tabular}{ll}
\\
	\noalign{\smallskip}\hline\noalign{\smallskip}
	(i) Laplacian & $\alpha=\theta=N-1$, $\beta=0 $ \\
(ii) $p$-Laplacian $(p\geq 2)$ & $\alpha=\theta=N-1$, $\beta=p-2$ \\
(iii) $k$-Hessian $(1\leq k\leq N)$ & $\alpha=N-k$, $\theta=N-1$, $\beta=k-1$\\
	\noalign{\smallskip}\hline\\
\end{tabular}
                                                   
\noindent where these operators act on the weighted Sobolev spaces 
\[ W^{1,p}_{\alpha,\theta}(0,R):=W^{1,p}((0,R),d\lambda_{\alpha},d\lambda_{\theta})\ \ \mbox{for}\ \ 0<R\leq\infty\]
 defined in section 2. The preposition \ref{P.0.1.0.0}, in section 2 [see Kufner-Opic \cite{Kufner-Opic}] gives us the following Sobolev type continuous embedding for $0<R<+\infty$
\[ W^{1,p}_{\alpha,\theta}(0,R)\hookrightarrow L_{\theta}^{q}(0,R) \, \ \mbox{if} \, \ 1\leq q\leq q^{\star},\ \alpha-\beta-1>0\, \ \mbox{and}\, \ p:=\beta+2,\]
and the number $q^{\star}:=\frac{(1+\theta)(\beta+2)}{\alpha-\beta-1}$ is the critical exponent associated with the weighted Sobolev space $W^{1,p}_{\alpha,\theta}(0,R)$. We would like to emphasize
 that continuity in the above embedding still hold in the following cases $\alpha-\beta-1=0$, $p=\beta+2$ with $1\leq q<\infty$.

As in the classical case, a function in $W^{1,p}_{\alpha,\theta}(0,R)$ (when $\alpha-(p-1)=0$) could have a local singularity, which proves that  \begin{math}W^{1,p}_{p-1,\theta}(0,R)\end{math} 
\begin{math}\not\subset L_{\theta}^{\infty}(0,R)\end{math}.  Motivated by this approach Oliveira and Do \'O \cite{Marcos do O - Oliveira} studied this embedding, and they proved some results on validity
 and attainability of the Trudinger-Moser inequality, for bounded domains see section \ref{sec:2}, Theorem \ref{P.0.1.5}, Theorem \ref{P.0.1.6.1} and Theorem \ref{P.0.1.7}. 

Our goal here is twofold: on the one hand, we prove a Trudinger-Moser type inequality for weighted Sobolev spaces involving fractional dimensions in the unbounded case $(0,\infty)$; and on the other
 hand, we discuss the existence of extremals functions in such inequalities. 

We will replace the constant $c_{\alpha,\theta}$ (wich depends on $\alpha$, $\theta$ and $R$) in Theorem \ref{P.0.1.5} by an uniform constant $d(\alpha,\theta,\mu)$ (wich depends on $\alpha$, 
$\theta$ and $\mu$), by replacing the Dirichlet norm with weight $\Vert u'\Vert_{L_{\alpha}^{p}}$ by the Sobolev norm with weights $\Vert u\Vert_{W^{1,p}_{\alpha,\theta}(0,\infty)}$, in the same spirit of
 the results stated in \cite{Li-Ruf} and \cite{Ruf}. Furthermore, we investigate the compactness on maximizing sequence for such inequalities in the same sense of the results stablished in \cite{Ishiwata}.

Let
\begin{equation*}
A_{p,\mu}(t)=e^{\mu t^{\frac{p}{p-1}}}-\sum_{j=0}^{ \lfloor p\rfloor-1}\frac{\mu^{j}}{j!}t^{\frac{p}{p-1}j},\, \ \mbox{with} \, \ \lfloor p\rfloor\, \ \mbox{the largest integer less than p.}
\end{equation*}
One of our main results is:
\begin{theorem}\label{0.1.9}
	Let $p\geq 2$, $\theta,\alpha\geq 0$ and $\mu>0$ be real numbers such that $\alpha-(p-1)=0$ and $\mu\leq\mu_{\alpha,\theta}$. Then there exists a constant $D(\theta,\alpha,\mu)$ which
	 depends only on $\theta$, $\alpha$ and $\mu$ such that
	\begin{equation}\label{0.1.10}
	\int_{0}^{\infty}A_{p,\mu}(\vert u(x)\vert)d\lambda_{\theta}(x)\leq D(\theta,\alpha,\mu) 
	\end{equation}
	for all $u\in W^{1,p}_{\alpha,\theta}(0,\infty)$ with $\Vert u\Vert_{W^{1,p}_{\alpha,\theta}(0,\infty)}=1$. Furthemore, the inequality (\ref{0.1.10}) fails if $\mu>\mu_{\alpha,\theta}$, that is, for any 
	$\mu>\mu_{\alpha,\theta}$ there exists a sequence $(u_{j})\subset W^{1,p}_{\alpha,\theta}(0,\infty)$ such that
	\begin{eqnarray*}
	\int_{0}^{\infty}A_{p,\mu}\left(\frac{\vert u_{j}(x)\vert}{\Vert u_{j}\Vert_{W^{1,p}_{\alpha,\theta}(0,\infty)}}\right)d\lambda_{\theta}(x)\to\infty \, \ \mbox{as}\,\ j\to\infty.
	\end{eqnarray*}  
\end{theorem}
To state our next results, we need to define the best constant associated with the inequality (\ref{0.1.10}), namely
\begin{eqnarray}\label{0.1.11}
& \displaystyle d(\theta,\alpha,\mu):=\sup_{0\neq u\in W^{1,p}_{\alpha,\theta}(0,\infty)}\int_{0}^{\infty}A_{p,\mu}\left(\frac{\vert u(x)\vert}{\Vert u\Vert_{W^{1,p}_{\alpha,\theta}(0,\infty)}}\right)d
\lambda_{\theta}(x),
\end{eqnarray}
where $\alpha-(p-1)=0$.
\begin{theorem}\label{0.1.12} Under the assumptions of Theorem \ref{0.1.9}, there exists a positive nonincreasing function $u$ in $W^{1,p}_{\alpha,\theta}(0,\infty)$ with $\Vert u\Vert_{W^{1,p}_{\alpha,
\theta}(0,\infty)}=1$ such that 
	\begin{equation*}\label{0.1.13}
	d(\theta,\alpha,\mu)=\int_{0}^{\infty}A_{p,\mu}(\vert u(x)\vert)d\lambda_{\theta}(x),
	\end{equation*}
	in the following cases:
	\begin{itemize}
		\item[$(i)$] $p\geq 3$ and $0<\mu<\mu_{\alpha,\theta}$, 
		\item[$(ii)$] $p=2$ and $\frac{2}{B(2,\theta)}<\mu<\mu_{\alpha,\theta}$.
	\end{itemize} 
	where 	\begin{equation*}\label{0.1.14}
	B(2,\theta)^{-1}:=\inf_{0\neq u\in W^{1,2}_{1,\theta}(0,\infty)}\frac{\left\Vert u'\right\Vert_{L_{1}^{2}(0,\infty)}^{2}\cdot\left\Vert u\right\Vert_{L_{\theta}^{2}(0,\infty)}^{2}}{\left\Vert u\right\Vert_{L_{\theta}
	^{4}(0,\infty)}^{4}}.
	\end{equation*}
\end{theorem}

\begin{theorem}\label{0.1.15}
	Let $p=2$, $\theta\geq 0$ and $\alpha=1$. Then there exists $\mu_{0}$ such that $d(\theta,\alpha,\mu)$ is not achieved for all $0<\mu<\mu_{0}$.
\end{theorem}

To prove (\ref{0.1.0.0}), Moser \cite{Moser} used the well known Schwarz Symmetrization arguments, which provides a radially symmetric function $u^{\#}$ defined on the ball $B_{R}(0)$, where 
$\mathcal{L}^{N}(\Omega)=\mathcal{L}^{N}(B_{R}(0))$ and all the balls $\{x\in B_{R}(0);u^{\#}(x)>t\}$ has the same $\mathcal{L}^{N}$ measure of the sets $\{x\in\Omega;u(x)>t\}$. Furthermore, $u^{\#}$
 satisfies the P\'olya-Szeg\"o inequality.
\begin{equation}\label{0.1.16}
\int_{B_{R}(0)}\vert\nabla u^{\#}\vert^{N}dx\leq\int_{\Omega}\vert\nabla u\vert^{N}dx.
\end{equation}
Thus the prove of (\ref{0.1.0.0}) was reduced to the subset of radially non-increasing symmetric functions. In our case, P\'olya-Szeg\"o inequality for $W^{1,p}_{\alpha,\theta}(0,\infty)$ was not available.
 That was one additional difficulty in this type of problem. See, for instance, \cite{Marcos do O - Oliveira}.

In this paper we present the half weighted Schwarz symmetrization with the goal of work around the problem. Thus, we will reduce again the Trudinger-Moser inequality to non-increasing functions.

The paper is organized as follows. In section $2$, we define some elements and present some previous results about Trudinger-Moser inequality on $W^{1,p}_{p-1,\theta}(0,R)$, where $R<\infty$. 
In section $3$, we prove a new P\'olya-Szeg\"o Principle on $W^{1,p}_{\alpha,\theta}$ using a new class of isoperimetric inequalities on $\mathbb{R}$ with respect to weights $\vert x\vert^{k}$. 
In section $4$, we establish the Trudinger-Moser inequality on $W^{1,p}_{\alpha,\theta}(0,\infty)$, under the assumptions of Theorem \ref{0.1.9}. In the section $5$, we obtain the Theorem \ref{0.1.12}
 studying the compactness of a maximizing sequence $(u_{n})$ for $(\ref{0.1.11})$. In the section $6$, we show the Theorem $\ref{0.1.15}$ proving that the functional $F(u)=\int_{0}^{\infty} A_{2,\mu}
 (\vert u(x)\vert)d\lambda_{\theta}(x)$ does not have criticals points on $\{u\in W^{1,2}_{1,\theta}(0,\infty): \Vert u\Vert_{W^{1,2}_{1,\theta}(0,\infty)}=1\}$. Finally, in the section $7$ we present a brief 
 discourse about Gagliardo-Nirenberg-Sobolev type inequality and we show that $2/B(2,\theta)<2\pi(1+\theta)$, thus the case $(ii)$ of the Theorem $\ref{0.1.12}$ makes sense.

\section{Basics definitions and previous results} \label{sec:2}

Let $0<R\leq +\infty$, $1\leq p<+\infty$ and $\theta\geq 0$. Let us denote by $L_{\theta}^{p}(0,R)$ the weighed Lebesque space defined as the set of all measuable functions $u$ on $(0,R)$ for which
\begin{eqnarray*}
\left\Vert u\right\Vert_{L_{\theta}^{p}(0,R)}:=\left[\int_{0}^{R}|u(x)|^{p}d\lambda_{\theta}(x)\right]^{1/p}<\infty, 
\end{eqnarray*} where 
\begin{eqnarray*}
d\lambda_{\theta}(x)=\omega_{\theta}x^{\theta}dx, \, \ \omega_{\theta}=\frac{2\pi^{\frac{1+\theta}{2}}}{\Gamma\left(\frac{1+\theta}{2}\right)}, \, \mbox{for all}\, \ \theta\geq 0, 
\end{eqnarray*}
with $\Gamma(x)=\int_{0}^{\infty}t^{x-1}e^{-t}dt$ the Gamma Function.
Besides, we denote by
\begin{eqnarray*}
W^{1,p}_{\alpha,\theta}(0,R):=\left\{u\in L_{\theta}^{p}(0,R); u' \in L_{\alpha}^{p}(0,R) \, \mbox{and} \, \lim_{x\to R^{-}} u(x)=0 \right\}
\end{eqnarray*}
and
\begin{eqnarray*}
\left\Vert u\right\Vert_{W^{1,p}_{\alpha,\theta}(0,R)}:=\left(\left\Vert u'\right\Vert_{L_{\alpha}^{p}(0,R)}^{p}+\left\Vert u\right\Vert_{L_{\theta}^{p}(0,R)}^{p}\right)^{\frac{1}{p}}.
\end{eqnarray*}

In the following proposition, see \cite{Kufner-Opic} for more details, we collect some embedding results for the weighted spaces $W^{1,p}_{\alpha,\theta}$, which will be used in this paper.

\begin{proposition}\label{P.0.1.0.0}
	Let $u: (0,R]\to\mathbb{R}$ be an absolutely continuous function. If $R<\infty$, $u(R)=0$ and 
	\begin{itemize}
		\item [$(1)$] for $1\leq\beta+2\leq q<\infty$ one has
		\begin{itemize}
			\item [$(a)$] $\alpha>\beta+1$, $\theta\geq \alpha\frac{q}{\beta+2}-q\frac{(\beta+1)}{\beta+2}-1$, or
			\item [$(b)$]$\alpha\leq\beta+1$, $\theta>-1$.
		\end{itemize}
		\item [$(2)$] for $1\leq q<\beta+2<\infty$ one has
		\begin{itemize}
			\item [$(c)$] $\alpha>\beta+1$, $\theta>\alpha\frac{q}{\beta+2}-q\frac{(\beta+1)}{\beta+2}-1$, or
			\item [$(d)$] $\alpha\leq\beta+1$, $\theta>-1$
		\end{itemize}
	\end{itemize}
	then \begin{equation*}\label{0.1.3}
	\left(\int_{0}^{R}\vert u\vert^{q}x^{\theta}dx\right)^{\frac{1}{q}}\leq C\left(\int_{0}^{R}\vert u'\vert^{\beta+2}x^{\alpha}dx\right)^{\frac{1}{\beta+2}},
	\end{equation*}
	where $C$ is a constant which does not depend on $u$.
\end{proposition}

Next, we present a results due to Oliveira and Do \'O \cite{Marcos do O - Oliveira}.

\begin{theorem}\label{P.0.1.5}
	Let $\alpha,\theta\geq 0$ and $p\geq 2$ be real numbers such that $\alpha-(p-1)=0$. Then there exists a constant $c_{\alpha,\theta}$ depending only on $\alpha,\theta$ and $R$ such that
	\begin{equation}\label{P.0.1.6}
	\sup_{u\in W^{1,p}_{\alpha,\theta}(0,R)}\int_{0}^{R}e^{\mu\left(\vert u\vert\right)^{\frac{p}{p-1}}}d\lambda_{\theta}(r)\left\{ 
	\begin{array}{cl}
	\leq c_{\alpha,\theta}, &\ \ \mbox{if}\  \, \ \mu\leq\mu_{\alpha,\theta}:= (1+\theta)\omega_{\alpha}^{\frac{1}{\alpha}}\\
	&\\
	=\infty, &\ \ \mbox{if}\ \, \ \mu>\mu_{\alpha,\theta},
	\end{array}
	\right.	
	\end{equation}
where $\Vert u'\Vert_{L_{\alpha}^{p}}=1$.
\end{theorem}
They also showed the existence of extremal functions for inequality $(\ref{P.0.1.6})$, as follows 

\begin{theorem}\label{P.0.1.6.1}
Under the assumptions of Theorem $\ref{P.0.1.5}$, there are extremal functions for $C_{\alpha,\theta,R}(\mu)$ when $\mu\leq\mu_{\alpha,\theta}$; that is, there exists $u\in W^{1,p}_{\alpha,\theta}(0,R)$
 such that
	\begin{eqnarray*}
	C_{\alpha,\theta,R}(\mu)=\int_{0}^{R}e^{\mu\vert u\vert^{\frac{p}{p-1}}}d\lambda_{\theta}(r),
	\end{eqnarray*}
where \begin{eqnarray*}
C_{\alpha,\theta,R}(\mu):=\sup_{u\in W^{1,p}_{\alpha,\theta}(0,\infty):\Vert u'\Vert_{L_{\alpha}^{p}}=1}\int_{0}^{R}e^{\mu\left(\vert u\vert\right)^{\frac{p}{p-1}}}d\lambda_{\theta}(r).
\end{eqnarray*}
\end{theorem}

In the same spirit of Adachi and Tanaka (see \cite{Adachi-Tanaka}), Oliveira and Do \'O showed the following result

\begin{theorem}\label{P.0.1.7}
Let $\theta,\alpha\geq 0$ and $p\geq 2$ be real numbers such that $\alpha-(p-1)=0$. Then for any $\mu\in (0,\mu_{\alpha,\theta})$ there is a constant $C_{\mu,p,\theta}$ depending only on $\mu$, $p$
 and $\theta$ such that
 \begin{equation}\label{0.1.7.0}
	\int_{0}^{\infty}A_{p,\mu}\left(\frac{\vert u(r)\vert}{\Vert u'\Vert_{L_{\alpha}^{p}(0,\infty)}}\right)d\lambda_{\theta}(r)\leq C_{\mu,p,\theta}\left(\frac{\Vert u\Vert_{L_{\theta}^{p}(0,\infty)}}{\Vert
	 u'\Vert_{L_{\alpha}^{p}(0,\infty)}}\right)^{p}
\end{equation}
for all $u\in W^{1,p}_{\alpha,\theta}(0,R)\backslash\{0\}$. Besides that, for any $\mu\geq\mu_{\alpha,\theta}$ there is a sequence $(u_{j})\subset W^{1,p}_{\alpha,\theta}(0,\infty)$ such that 
$\Vert u_{j}'\Vert_{L_{\alpha}^{p}(0,\infty)}=1$ and 
\begin{eqnarray*}
\frac{1}{\Vert u_{j}'\Vert_{L_{\alpha}^{p}(0,\infty)}}\int_{0}^{\infty}A_{p,\mu}\left(\vert u_{j}(r)\vert\right)d\lambda_{\theta}(r)\to\infty \, \ \mbox{as}\, \ j\to\infty.
\end{eqnarray*}
Where
\begin{equation*}\label{0.1.8}
A_{p,\mu}(t)=e^{\mu t^{\frac{p}{p-1}}}-\sum_{j=0}^{\lfloor p\rfloor-1}\frac{\mu^{j}}{j!} t^{\frac{p}{p-1}j},\, \ \mbox{with} \, \ \lfloor p\rfloor\, \ \mbox{is the largest integer less than p.}
\end{equation*}
\end{theorem}
As mentioned in the Introduction, Ishiwata \cite{Ishiwata} studied the attainability of $d_{N,\mu}$ $($\ref{0.1.1.0}$)$ in the classical case. He emphasized the importance of evaluate vanishing behaviour
 on maximizing sequence in unbounded case. Next, the main results in \cite{Ishiwata} are presented. 

\begin{theorem}\label{P.0.1.0.1}
	Let $N\geq 2$ and \begin{eqnarray*}
	B_{2}:=\sup_{0\neq\psi\in W^{1,2}(\mathbb{R}^{2})}\frac{\left\Vert \psi\right\Vert_{L^{4}}^{4}}{\left\Vert \nabla \psi\right\Vert_{L^{2}}^{2}\left\Vert \psi\right\Vert_{L^{2}}^{2}}.
	\end{eqnarray*}
	Then $d_{N,\mu}$ is attained for $0<\mu<\mu_{N}$ if $N\geq 3$ and for $2/B_{2}<\mu\leq\mu_{2}=4\pi$ if $N=2$.
\end{theorem}

\begin{theorem}\label{P.0.1.0.2}
	Let $N=2$. If $\mu\ll 1$, then $d_{2,\mu}$ is not attained.
\end{theorem}

\section{P{\'o}lya-Szeg{\"o} Principle on $W^{1.p}_{\alpha,\theta}$}\label{sec:3}

As mentioned in the introduction, we are going to define a half weighted Schwarz symmetrization to prove a P\'olya-Szego Principle, see the inequality $(\ref{0.1.16})$.

We define the measure $\mu_{l}$ by $d\mu_{l}(x)= \vert x\vert^{l}dx$. Besides, if $M\subset\mathbb{R}$ is a measurable set with finite $\mu_{l}$-measure, then let $M^{*}$ denote the interval $(0,R)$
 such that 
\begin{eqnarray*}
\mu_{l}((0,R))=\mu_{l}(M).
\end{eqnarray*}
Further, if $u:\mathbb{R}\longrightarrow\mathbb{R}$ is a measurable function such that 
\begin{eqnarray*}
\mu_{l}\left(\{y\in\mathbb{R};\left\vert u(y)\right\vert>t\}\right)<\infty \, \ \mbox{for all}\,\ t>0,
\end{eqnarray*}
then let $u^{*}$ denote the half weighted Schawarz symmetrization of $u$, or in short, the half $\mu_{l}$-symmetrization of $u$, given by
\begin{equation*}
u^{*}(x)=\sup\left\{t\geq 0; \mu_{l}\left(\left\{y\in \mathbb{R};\left\vert u(y)\right\vert>t\right\}\right)>\mu_{l}(0,x)\right\},
\end{equation*}
for every $x>0$.

\begin{remark}
	The word ``half'' appears here because our symmetrization is a little bit different in three aspects:
	\itemize
	\item[$(i)$] it is defined on $(0,\infty)$;
	\item[$(ii)$] we are comparing the distribution $\rho(t)
	:=\mu_{l}\left(\left\{y\in\mathbb{R};\left\vert u(y)\right\vert>t\right\}\right)$ with the measure of $(0,x)$, instead $B_{\vert x\vert}(0)$;
	\item[$(iii)$] the set $M^{*}$ is a semi ball with the same measure of $M$, instead a ball.
\end{remark}

We will carry out the proof of the next result based on Isoperimetric Inequality on $\mathbb{R}$ with weight $\left\vert x\right\vert^{k}$ [see \cite{Alvino},Theorem 6.1]. Besides, It worth noting that the
 Theorem $8.1$ in \cite{Alvino} do not cover the case $k<l+1$ when $N=1$. For negative values of $k$, the proof is a consequence of the well-known Hardy-Littlewood inequlaity. See also Cabr\'e and
  Ros-Oton \cite{Cabre-Ros-Oton} for monomial weights, and Talenti \cite{Talenti} for some cases when $N\geq2$.

\begin{theorem}\label{2} Let $k,l$ be real numbers satisfying $0<k\leq l+1$. Besides, let $1\leq p<\infty$ and $m:=pk+(1-p)l$. Then there holds
	\begin{equation}\label{2.1}
	\int_{0}^{\infty}\left\vert u'\right\vert^{p}\left\vert x\right\vert^{pk+(1-p)l}dx\geq\int_{0}^{\infty}\left\vert (u^{*})' \right\vert^{p}\left\vert x\right\vert^{pk+(1-p)l}dx,
	\end{equation}
	for every $u\in W^{1,p}_{l,m}(0, \infty)$, where $u^{*}$ denotes the half $\mu_{l}$-symmetrization of $u$.
\end{theorem}

\begin{proof}
	Observe that it is sufficient to consider $u$ a non-negative function. Let
	\begin{eqnarray*}
	&\displaystyle I:=\int_{0}^{\infty}\left\vert u'\right\vert^{p}\left\vert x\right\vert^{pk+(1-p)l}dx \, \ \mbox{and}\\
	& \displaystyle I^{*}:=\int_{0}^{\infty}\left\vert (u^{*})'\right\vert^{p}\left\vert x\right\vert^{pk+(1-p)l}dx.
	\end{eqnarray*}
	The Coarea Formula holds 
	
	\begin{eqnarray*}
	&\displaystyle I:=\int_{0}^{\infty}\int_{u=t}\left\vert u'\right\vert^{p-1}\left\vert x\right\vert^{pk+(1-p)l}d\mathcal{H}^{0}(x)dt\, \ \mbox{and}\\
	&\displaystyle I^{*}:=\int_{0}^{\infty}\int_{u^{*}=t}\left\vert (u^{*})'\right\vert^{p-1}\left\vert x\right\vert^{pk+(1-p)l}d\mathcal{H}^{0}(x)dt.
	\end{eqnarray*}
	
If $p=1$, we get
\begin{align*}
	&\displaystyle I:=\int_{0}^{\infty}\int_{u=t}\left\vert x\right\vert^{k}d\mathcal{H}^{0}(x)dt\, \ \mbox{and}\nonumber\\
	&\displaystyle I^{*}:=\int_{0}^{\infty}\int_{u^{*}=t}\left\vert x\right\vert^{k}d\mathcal{H}^{0}(x)dt, 
\end{align*}
hence, we obtain from Isoperimetric Inequality on $\mathbb{R}$ with weight $\left\vert x\right\vert^{k}$ [see \cite{Alvino}, Theorem 6.1] and definition of $u^{\star}$ that
   \begin{align*}
   \int_{u=t}\left\vert x\right\vert^{k}d\mathcal{H}^{0}(x)\geq\int_{u^{*}=t}\left\vert x\right\vert^{k}d\mathcal{H}^{0}(x).
   \end{align*}
  Therefore, $I\geq I^{\star}$ when $p=1$.
  
Now, asssume that $1<p<\infty$.	By Holder's Inequality we have 
\begin{eqnarray*}
	\int_{u=t}\left\vert x\right\vert^{k} dH^{0}(x)&\leq&\left(\int_{u=t}\left\vert x\right\vert^{kp+(1-p)l}\left\vert u'\right\vert^{p-1}d\mathcal{H}^{0}(x)\right)^{\frac{1}{p}}
\left(\int_{u=t}\frac{\left\vert x\right\vert^{l}}{\left\vert u'\right\vert}d\mathcal{H}^{0}(x)\right)^{\frac{p-1}{p}}
	\end{eqnarray*}
	for $a.e$ $t\in [0, \infty)$, thus we get
	\begin{equation}\label{2.2}
	I\geq\int_{0}^{\infty}\left(\int_{u=t}\left\vert x\right\vert^{k}d\mathcal{H}^{0}(x)\right)^{p}\left(\int_{u=t}\frac{\left\vert x\right\vert^{l}}{\left\vert u'\right\vert}d\mathcal{H}^{0}(x)\right)^{1-p}dt.
	\end{equation}
 Since that $\vert(u^{\star})'\vert$ and $\vert x\vert$ are constants along of $\{u^{\star}=t\}$, hence, for $u^{*}$ we obtain the equality, i.e, 
	\begin{equation}\label{2.3}
	I^{*}=\int_{0}^{\infty}\left(\int_{u^{*}=t}\left\vert x\right\vert^{k}d\mathcal{H}^{0}(x)\right)^{p}\left(\int_{u^{*}=t}\frac{\left
		\vert x\right\vert^{l}}{\left\vert (u^{\star})'\right\vert}d\mathcal{H}^{0}(x)\right)^{1-p}dt.
	\end{equation}	
	In addition, by definition of $u^{*}$, we have
	\begin{eqnarray*}
	\displaystyle \int_{u>t}\left\vert x\right\vert^{l}dx=\int_{u^{*}>t}\left\vert x\right\vert^{l}dx,
	\end{eqnarray*}
	and as a consequence of Coarea Formula we get
	\begin{equation}\label{2.4}
	\int_{u=t}\frac{\left\vert x\right\vert^{l}}{\left\vert u'\right\vert}d\mathcal{H}^{0}(x)=	\int_{u^{*}=t}\frac{\left\vert x\right\vert^{l}}{\left\vert (u^{*})'\right\vert}d\mathcal{H}^{0}(x),
	\end{equation}
	for $a.e$ $t\in[0, \infty)$, that is sometimes called Fleming - Rishel's Formula.
	
Again, by Isoperimetric Inequality on $\mathbb{R}$ with weight $\left\vert x\right\vert^{k}$ [see \cite{Alvino},Theorem 6.1] and the definition of $u^{*}$ we obtain
	\begin{equation}\label{2.5}
	\int_{u=t}\left\vert x\right\vert^{k}d\mathcal{H}^{0}(x)\geq\int_{u^{*}=t}\left\vert x\right\vert^{k}d\mathcal{H}^{0}(x).	
	\end{equation}
	Therefore, from (\ref{2.2}), (\ref{2.3}), (\ref{2.4}), and (\ref{2.5}) we have
	\begin{eqnarray*}
	I\geq I^{*},
	\end{eqnarray*}  
	thus, (\ref{2.1}) follows.
\end{proof}

\section{Trudinger-Moser inequality on $W^{1,p}_{\alpha,\theta}(0,\infty)$}\label{sec:4}

In this section, we establish a Trudinger-Moser type inequality on $W^{1,p}_{\alpha,\theta}(0,\infty)$ (Theorem \ref{0.1.9}) via  the P{\'o}lya-Szeg{\"o} Principle presented in section \ref{sec:3}.
 
\begin{lemma}\label{1.7} \begin{enumerate}
\item [(i)]Let $u$ be a function in $W^{1,p}_{\alpha,\theta}(0,\infty)$. Then
\begin{equation}\label{1.7.1}
\vert u(x)\vert^{p}\leq p\omega_{\theta}^{-\frac{p-1}{p}}\omega_{\alpha}^{-\frac{1}{p}}x^{-\frac{(p-1)\theta+\alpha}{p}}\left\Vert u\right\Vert_{L_{\theta}^{p}(0,\infty)}^{p-1}\left\Vert u'\right\Vert_{L_{\alpha}^{p}
(0,\infty)}\ \ \text{for all}\ \ x>0.
\end{equation}
 Consequently, the embedding $W^{1,p}_{\alpha,\theta}(0,\infty)\hookrightarrow L^q_\theta(0,\infty)$ is compact for all $q$ satisfying
 \[\dfrac{p^{2}(1+\theta)}{(p-1)\theta+\alpha}\leq q<\dfrac{p(1+\theta)}{\alpha-(p-1)}:=p^{\star},\] 
	where $\alpha\geq(p-1)$ and $\alpha\leq p+\theta$.
	
\item [(ii)] Let $u\in L_{\theta}^{p}(0,R)$ a nonincreasing function, then 
\begin{equation}\label{1.5}
	|u(x)|\leq \left(\frac{1+\theta}{\omega_{\theta}x^{1+\theta}}\right)^{1/p}\left[\int_{0}^{R}|u(s)|^{p}d\lambda_{\theta}(s)\right]^{1/p}, \, \ \mbox{for all}\, \ 0<x< R.	
\end{equation}	
Hence, if $(u_{n})\subset W^{1,p}_{\alpha,\theta}(0,\infty)$ is a nonincreasing sequence converging weakly to $u$ in $W^{1,p}_{\alpha,\theta}(0,\infty)$, then $u_{n}\to u$ strongly in 
$L^{q}_{\theta}(0,\infty)$, for each $p<q<p^{\star}$ $(\alpha\geq p-1)$.
\end{enumerate}
\end{lemma}
\begin{proof} It is easy to check out (\ref{1.5}) from a nonincreasing function. Then, we will do only (\ref{1.7.1}).  
	
For every $0<x<y$ we have
\begin{eqnarray*}
	\left\vert u(x)\right\vert^{p}\leq \left\vert u(y)\right\vert^{p}+p\int_{x}^{y}\left\vert u(t)\right\vert^{p-1}\left\vert u'(t)\right\vert dt.
\end{eqnarray*}
By Holder Inequality and $\displaystyle\lim_{y\to\infty} u(y)=0$, we get
\begin{eqnarray*}
	\begin{array}{l}
	\displaystyle\left\vert u(x)\right\vert^{p}\leq
	p\int_{x}^{\infty}\left\vert u(t)\right\vert^{p-1}\left\vert u'(t)\right\vert dt\nonumber\\
	\quad\displaystyle\leq  p\omega_{\theta}^{-\frac{p-1}{p}}\omega_{\alpha}^{-\frac{1}{p}}x^{-\frac{p-1}{p}(1+\theta)}\left(\int_{0}^{\infty}\left\vert u(t)\right\vert^{p}d\lambda_{\theta}(t)\right)^{\frac{p-1}{p}}
	\cdot \left(\int_{0}^{\infty}\left\vert u'(t)\right\vert^{p}d\lambda_{\alpha}(t)\right)^{\frac{1}{p}}
	\end{array}
\end{eqnarray*}
which proves (\ref{1.7.1}).
\end{proof}

The next remark will be used in the proof of Theorem \ref{0.1.9}.
\begin{remark}\label{1.7.2}
By inequality (\ref{1.7.1}), we have  $\vert u(x)\vert\leq 1$, for all \begin{align*}x\geq\left(\dfrac{p}{\omega_{\theta}^{\frac{p}{p-1}}\omega_{\alpha}^{\frac{1}{p}}}\right)^{\frac{p}{(p-1)(1+
\theta)}}:=a_{0}\end{align*}
whenever $u\in W^{1,p}_{\alpha,\theta}(0,\infty)$ with $\Vert u\Vert_{W^{1,p}_{\alpha,\theta}(0, \infty)}\leq 1$ and $\alpha-(p-1)=0$. It is worth noting that $a_{0}$ depends only on $p$, and $\theta$. 
\end{remark}

\noindent\textbf{Proof of Theorem \ref{0.1.9}}

 We can assume by Theorem \ref{2} that $u$ is a nonincreasing positive function on $(0,\infty)$.

Let $a\geq a_{0}$ (see Remark \ref{1.7.2}) to be chosen later. Next, we divide the integral at (\ref{0.1.10}) in two parts, that is,
\begin{equation}\label{1.8}
\int_{0}^{\infty}A_{p,\mu}(|u(x)|)d\lambda_{\theta}(x)=\int_{0}^{a}A_{p,\mu}(|u|)d\lambda_{\theta}(x)+\int_{a}^{\infty}A_{p,\mu}(\vert u\vert)d\lambda_{\theta}(x).
\end{equation}
By Lemma \ref{1.7}, the second part at (\ref{1.8}) can be to estimated. Indeed, we have
\begin{eqnarray*}
\int_{a}^{\infty}A_{p,\mu}(\vert u\vert)d\lambda_{\theta}(x)=\sum_{j=\lfloor p\rfloor}^{\infty}\frac{\mu^{j}}{j!}\int_{a}^{\infty}\vert u\vert^{\frac{p}{p-1}j}r^{\theta}\omega_{\theta}dr.
\end{eqnarray*}

We obtain by Lemma \ref{1.7} and Remark \ref{1.7.2}
\begin{eqnarray}\label{1.8.1}
\displaystyle\int_{a}^{\infty}A_{p,\mu}(\vert u\vert)d\lambda_{\theta}(x) & = & \displaystyle\sum_{j=\lfloor p\rfloor}^{\infty}\frac{\mu^{j}}{j!}\int_{a}^{\infty}\vert u\vert^{\frac{p}{p-1}j}r^{\theta}\omega_{\theta}dr
 \nonumber\\
 & \leq &\omega_{\theta}\frac{\mu^{\lfloor p\rfloor}}{\lfloor p\rfloor!}\int_{0}^{\infty}\vert u\vert^{p}r^{\theta}dr  \nonumber \\
 & + &\omega_{\theta}\sum_{j= \lfloor p\rfloor+1}^{\infty}\frac{\mu^{j}(1+\theta)^{\frac{j}{p-1}}}{j!\omega_{\theta}^{\frac{j}{p-1}}}\left[\omega_{\theta}\int_{0}^{\infty}\vert u\vert^{p}r^{\theta}dr\right]^{\frac{j}
 {p-1}}\nonumber\\
 &\cdot&\int_{a}^{\infty}r^{\theta-\frac{(1+\theta)j}{p-1}}dr\nonumber\\
 =&\dfrac{\mu^{\lfloor p\rfloor}}{\lfloor p\rfloor!}&+\sum_{j= \lfloor p\rfloor+1}^{\infty}\frac{\mu^{j}(1+\theta)^{\frac{j}{p-1}}(p-1)\omega_{\theta}}{j!\omega_{\theta}^{j/p-1}(1+\theta)(j-(p-1))a^{\frac{(1+\theta)j}
 {p-1}}} 
\end{eqnarray}

To estimate the first part at (\ref{1.8}), let
\begin{eqnarray*}
v(r)=\left\{	\begin{array}{cc}
		u(r)-u(a), & 0<r\leq a \\
		0, & r\geq a\\
	\end{array}
 \right.
\end{eqnarray*}
Note that if $1<q\leq 2$ and $b\geq 0$, we have $(x+b)^{q}\leq \vert x\vert^{q}+qb^{q-1}x +b^{q}$ for all $x\geq -b$.
Then, by Lemma \ref{1.7} we obtain
\begin{eqnarray}\label{1.10}
u(r)^{\frac{p}{p-1}}&\leq& v(r)^{\frac{p}{p-1}}+\frac{p}{p-1}v(r)^{\frac{1}{p-1}}u(a)+u(a)^{\frac{p}{p-1}}\nonumber	\\
& \leq & v(r)^{\frac{p}{p-1}}+v(r)^{\frac{p}{p-1}}u(a)^{p}+u(a)^{\frac{p}{p-1}}+\frac{1}{(p-1)^{1/p-1}}\nonumber\\
&\leq& v(r)^{\frac{p}{p-1}}\left[1+\frac{1+\theta}{a^{1+\theta}\omega_{\theta}}\left(\omega_{\theta}\int_{0}^{\infty}|u|^{p}r^{\theta}dr\right)\right]+\left(\frac{1+\theta}{a^{1+\theta}\omega_{\theta}}\right)^{1/
p-1}\nonumber\\
&+&\frac{1}{(p-1)^{1/p-1}}\nonumber\\
&:=& v(r)^{\frac{p}{p-1}}\left[1+\frac{1+\theta}{a^{1+\theta}\omega_{\theta}}\left(\omega_{\theta}\int_{0}^{\infty}|u|^{p}r^{\theta}dr\right)\right]+d(a).
\end{eqnarray}
	Hence
\begin{eqnarray*}
u(r)&\leq& v(r)\left[1+\frac{1+\theta}{a^{1+\theta}\omega_{\theta}}\left(\omega_{\theta}\int_{0}^{\infty}|u|^{p}r^{\theta}dr\right)\right]^{\frac{p-1}{p}}+d(a)^{\frac{p-1}{p}}\nonumber\\
&:=&w(r)+d(a)^{\frac{p}{p-1}},
\end{eqnarray*} thus
\begin{eqnarray}\label{1.9}
\omega_{\alpha}\int_{0}^{a}|w'|^{p}r^{\alpha}dr&=&\omega_{\alpha}\int_{0}^{a}|u'|^{p}\left[1+\frac{1+\theta}{a^{1+\theta}\omega_{\theta}}\left(\omega_{\theta}\int_{0}^{\infty}|u|^{p}r^{\theta}dr\right)
\right]^{p-1}r^{\alpha}dr\nonumber\\
&=&\left[1+\frac{1+\theta}{a^{1+\theta}\omega_{\theta}}\left(\omega_{\theta}\int_{0}^{\infty}|u|^{p}r^{\theta}dr\right)\right]^{p-1}\omega_{\alpha}\int_{0}^{a}|u'|^{p}r^{\alpha}dr\nonumber\\
&\leq& \left[1+\frac{1+\theta}{a^{1+\theta}\omega_{\theta}}\left(\omega_{\theta}\int_{0}^{\infty}|u|^{p}r^{\theta}dr\right)\right]^{p-1}\left[1-\omega_{\theta}\int_{0}^{\infty}|u|^{p}r^{\theta}dr\right]\nonumber\\
&\leq& 1
\end{eqnarray}
where in the last inequality we used that the function $f:[0,1]\to \mathbb{R}$ defined by $f(t)=(1+\gamma t)^{p-1}(1-t)-1$ is non-positive for any $\gamma$ fixed in the interval $(0,1/(p-1))$ and 
consequentely the inequality (\ref{1.9}) is valid with
\begin{eqnarray*}
\left(\frac{(p-1)(1+\theta)}{\omega_{\theta}}\right)^{1/(1+\theta)}\leq a<\infty.
\end{eqnarray*} 
Next, from (\ref{1.10}) we have
\begin{eqnarray}\label{1.9.1}
\int_{0}^{a}A_{p,\mu}(\vert u(x)\vert)d\lambda_{\theta}(x)&\leq&\omega_{\theta}\int_{0}^{a}e^{\mu\vert u\vert^{\frac{p}{p-1}}}r^{\theta}dr\nonumber\\
&\leq&\omega_{\theta}\int_{0}^{a}e^{\mu\vert w\vert^{\frac{p}{p-1}}}r^{\theta}dr+\omega_{\theta}\int_{0}^{a}e^{d(a)}r^{\theta}dr.
\end{eqnarray}
We combine (\ref{1.8.1}), (\ref{1.9}), (\ref{1.9.1}) and Theorem \ref{P.0.1.5} to conclude the first part of the proof of the theorem.

For the second part, we are going to do the changing of variable as in \cite{Marcos do O - Oliveira}. We define $w(t)=\omega_{\alpha}^{\frac{1}{\alpha+1}}(1+\theta)^{\frac{\alpha}{1+\alpha}}u(Re^{-
\frac{t}{1+\theta}})$ for all $u\in W^{1,p}_{\alpha,\theta}(0,R)$, where $\alpha-(p-1)=0$. Then, we get

\begin{eqnarray}\label{1.9.2}
\int_{0}^{R}\vert u'(r)\vert^{p}d\lambda_{\alpha}(r)=\int_{0}^{\infty}\vert w'(t)\vert^{p}dt,
\end{eqnarray}
\begin{eqnarray}\label{1.9.3}
\int_{0}^{R}\vert u(r)\vert^{p}d\lambda_{\theta}(r)=\frac{R^{1+\theta}\omega_{\theta}}{(1+\theta)^{p}\omega_{\alpha}}\int_{0}^{\infty}\vert w(t)\vert^{p}e^{-t}dt
\end{eqnarray}
and \begin{eqnarray}\label{1.9.4}
\int_{0}^{R}e^{\vert u\vert^{\frac{p}{p-1}}}d\lambda_{\theta}(r)=\frac{\omega_{\theta} R^{1+\theta}}{1+\theta}\int_{0}^{\infty}e^{\frac{\mu}{\mu_{\alpha,\theta}}\vert w\vert^{\frac{p}{p-1}}-t}dt.
\end{eqnarray}
We consider Moser's functions
\begin{eqnarray*}
	w_{j}(t)=\left\{\begin{array}{cc}
		\frac{t}{j^{\frac{1}{p}}} & 0\leq t\leq j\\
		j^{\frac{p-1}{p}} & t\geq j.
	\end{array} 
\right.
\end{eqnarray*}

Hence, we obtain  from (\ref{1.9.2}), (\ref{1.9.3}) and (\ref{1.9.4}) that
\begin{eqnarray*}
\int_{0}^{R}e^{\left(\frac{\vert u_{j}\vert}{\Vert u_{j}\Vert_{W^{1,p}_{\alpha,\theta}(0,R)}}\right)^{\frac{p}{p-1}}}d\lambda_{\theta}(r)&=&\frac{\omega_{\theta} R^{1+\theta}}{1+\theta}\int_{0}^{\infty}
e^{\frac{\mu\vert w_{j}\vert^{\frac{p}{p-1}}}{\mu_{\alpha,\theta}\left(1+\rho(\alpha,\theta,R)a_{j}\right)^{\frac{1}{p-1}}}-t}dt\\
&\geq& \large{e}^{\left(\frac{\mu}{\mu_{\alpha,\theta}(1+\rho\left(\alpha,\theta,R\right)a_{j})^{\frac{1}{p-1}}}-1\right)j},
\end{eqnarray*} where $\rho(\alpha,\theta,R)=\frac{R^{1+\theta}\omega_{\theta}}{(1+\theta)^{p}\omega_{\alpha}}$, $a_{j}=\frac{1}{j}\displaystyle\int_{0}^{j}e^{-t}t^{p}dt+j^{p-1}e^{-j}$ and $w_{j}(t)=
\omega_{\alpha}^{\frac{1}{1+\alpha}}(1+\theta)^{\frac{\alpha}{\alpha+1}}u_{j}(Re^{-\frac{t}{(1+\theta)}})$.
Thus, if $\mu>\mu_{\alpha,\theta}$  \begin{eqnarray*}
\lim_{j\to\infty}\int_{0}^{R}e^{\left(\frac{\vert u_{j}\vert}{\Vert u_{j}\Vert_{W^{1,p}_{\alpha,\theta}(0,R)}}\right)^{\frac{p}{p-1}}}d\lambda_{\theta}(r)&\geq&\lim_{j\to\infty}e^{\left(\frac{\mu}{\mu_{\alpha,\theta}
(1+\rho\left(\alpha,\theta,R\right)a_{j})^{\frac{1}{p-1}}}-1\right)j}\\
&=&+\infty.
\end{eqnarray*} 
Which concludes the theorem.

\section{Proof of the Theorem \ref{0.1.12}}
\label{sec:5}
In this section, we are going to show the Theorem \ref{0.1.12}. To show the attainability, we study the maximizing sequence to (\ref{0.1.11}). Throughout this section we assume ( via Lemma \ref{2}) that
 $(u_{n})$ is a non-increasing positive maximizing sequence to (\ref{0.1.11}). Besides, assume

\begin{equation*}
u_{n}\rightharpoonup u \, \ \mbox{in}\, \ W^{1,p}_{\alpha,\theta}(0,\infty),\ \ \mbox{where}\ \ \alpha-(p-1)=0.
\end{equation*}

We begin with
\begin{lemma}\label{2.1.0} Let $0<\mu<\mu_{\alpha,\theta}$. Then, we have
	\begin{align}\label{2.1.1}
&\displaystyle\int_{0}^{\infty}A_{p,\mu}\left(\left\vert u_{n}\right\vert\right)-\frac{\mu^{\lfloor p\rfloor}}{\lfloor p\rfloor!}\left\vert u_{n}\right\vert^{\frac{p\lfloor p\rfloor}{p-1}} d\lambda_{\theta}-\int_{0}^{\infty}
A_{p,\mu}\left(\left\vert u\right\vert\right)-\frac{\mu^{\lfloor p\rfloor}}{\lfloor p\rfloor!}\left\vert u\right\vert^{\frac{p\lfloor p\rfloor}{p-1}} d\lambda_{\theta}\to 0  \nonumber \\
&\text{as}\, \ n\to\infty.  \end{align}
	
\end{lemma} 
\begin{proof}
	We rewritten (\ref{2.1.1}) as follows
	\begin{eqnarray*}\label{2.1.2}
	\int_{0}^{\infty}B_{\lfloor p\rfloor+1,\mu}\left(\left\vert u_{n}\right\vert\right)\ d\lambda_{\theta}-\int_{0}^{\infty}B_{\lfloor p\rfloor+1,\mu}\left(\left\vert u\right\vert\right) d\lambda_{\theta}\to
	 0\end{eqnarray*}
	as $n\to\infty$,
\end{proof}
where \begin{align*}
	B_{k,\mu}(t):=\sum_{j=k}^{\infty}\dfrac{\mu^{j}}{j!}t^{\frac{p}{p-1}j}, \, \ \text{where}\, \ k\in\mathbb{N} \, \ \text{and}\, \ t\in [0,\infty).
\end{align*}
It follows from Mean Value Theorem and convexity of $B_{\lfloor p\rfloor+1,\mu}$ that
\begin{align}\label{2.1.3}
&\left\vert B_{\lfloor p\rfloor+1,\mu}(u_{n}(x))-B_{\lfloor p\rfloor+1,\mu}(u(x))\right\vert\nonumber\\
&\leq \left(B_{\lfloor p\rfloor+1,\mu}\right)'(\gamma_{n}(x)u_{n}(x)+(1-\gamma_{n}(x)u(x))\cdot\left\vert u_{n}(x)-u(x)\right\vert\nonumber\\
&=\mu\frac{p}{p-1}\left\vert \gamma_{n}(x)u_{n}(x)+(1-\gamma_{n}(x)u(x)\right\vert^{\frac{1}{p-1}}\nonumber\\
 &\cdot B_{\lfloor p\rfloor,\mu}(\gamma_{n}(x)u_{n}(x)+(1-\gamma_{n}(x))u(x))\cdot\left\vert u_{n}(x)-u(x)\right\vert\nonumber\\
&\leq\mu\frac{p}{p-1}\left\vert \gamma_{n}(x)u_{n}(x)+(1-\gamma_{n}(x)u(x)\right\vert^{\frac{1}{p-1}}\nonumber\\
&\cdot\left[\gamma_{n}(x)B_{\lfloor p\rfloor,\mu}(u_{n}(x))+(1-\gamma_{n}(x))B_{\lfloor p\rfloor,\mu}(u(x))\right]\cdot\left\vert u_{n}(x)-u(x)\right\vert \nonumber\\
&\leq\mu\frac{p}{p-1}\left\vert \gamma_{n}(x)u_{n}(x)+(1-\gamma_{n}(x)u(x)\right\vert^{\frac{1}{p-1}}\cdot \left[A_{p,\mu}(u_{n}(x))+A_{p,\mu}(u(x))\right]\nonumber\\
&\cdot\left\vert u_{n}(x)-u(x)\right\vert
\end{align}
Now, by H\"older's and Minkowski's Inequalities, and (\ref{2.1.3}) we get

\begin{align}\label{2.1.4}
&\left\vert\int_{0}^{\infty}B_{\lfloor p\rfloor+1,\mu}\left(\left\vert u_{n}\right\vert\right) d\lambda_{\theta}-\int_{0}^{\infty}B_{\lfloor p\rfloor+1,\mu}\left(\left\vert u\right\vert\right) d\lambda_{\theta}\right\vert
\nonumber\\
&\leq\mu\frac{p}{p-1}\left(\int_{0}^{\infty}\left\vert \gamma_{n}(x)u_{n}(x)+(1-\gamma_{n}(x))u(x)\right\vert^{\frac{r}{p-1}}d\lambda_{\theta}(x)\right)^{\frac{1}{r}}\nonumber\\
&\cdot\left(\int_{0}^{\infty}\left[A_{ p,\mu}(u_{n}(x))+A_{ p,\mu}(u(x))\right]^{q}d\lambda_{\theta}(x)\right)^{\frac{1}{q}}\left(\int_{0}^{\infty}\left\vert u_{n}(x)-u(x)\right\vert^{t}d\lambda_{\theta}(x)
\right)^{\frac{1}{t}}\nonumber\\
&\leq\mu\frac{p}{p-1}\Vert u_{n}\Vert_{L_{\theta}^{\frac{r}{p-1}}(0,\infty)}^{\frac{1}{p-1}}\Vert u\Vert_{L_{\theta}^{\frac{r}{p-1}}(0,\infty)}^{\frac{1}{p-1}}\left(\int_{0}^{\infty}\left(A_{ p,\mu}(u_{n}(x))\right)^{q}d
\lambda_{\theta}(x)\right)^{\frac{1}{q}}\nonumber\\
&\cdot\left(\int_{0}^{\infty}\left(A_{ p,\mu}(u(x))\right)^{q}d\lambda_{\theta}(x)\right)^{\frac{1}{q}}\cdot\Vert u_{n}-u\Vert_{L_{\theta}^{t}(0,\infty)}\nonumber\\
&\leq\mu\frac{p}{p-1}\Vert u_{n}\Vert_{L_{\theta}^{\frac{r}{p-1}}(0,\infty)}^{\frac{1}{p-1}}\Vert u\Vert_{L_{\theta}^{\frac{r}{p-1}}(0,\infty)}^{\frac{1}{p-1}}\left(\int_{0}^{\infty}A_{p,q\mu}(u_{n}(x))d
\lambda_{\theta}(x)\right)^{\frac{1}{q}}\nonumber\\
&\cdot\left(\int_{0}^{\infty}A_{p,q\mu}(u(x))d\lambda_{\theta}(x)\right)^{\frac{1}{q}}\Vert u_{n}-u\Vert_{L_{\theta}^{t}(0,\infty)},
 \end{align}
 where $q,r,t>1$ are real numbers satisfying $\frac{1}{r}+\frac{1}{q}+\frac{1}{t}=1$, $q\mu<\mu_{\alpha,\theta}$, $\frac{r}{p-1}\geq p$ and $t>\frac{p^{2}}{p-1}$. Besides, in the last inequality at \ref{2.1.4})
  we used the following inequality
 \begin{eqnarray*}
 \left(e^{\mu t^{\frac{p}{p-1}}}-\sum_{j=0}^{\lfloor p\rfloor-1}\frac{\mu^{j}}{j!}t^{\frac{p}{p-1}j}\right)^{q}\leq e^{q\mu t^{\frac{p}{p-1}}}-\sum_{j=0}^{\lfloor p\rfloor-1}\frac{(q\mu)^{j}}{j!}t^{\frac{p}{p-1}j}.
 \end{eqnarray*}
Therefore, from (\ref{2.1.4}), Lemma \ref{1.7} and compactness embedding we conclude the proof of the Lemma.

To continue the study of the maximizing sequence $(u_{n})$ based on the concentration-compactness type argument, we analyze the possibility of a lack of compactness which is called vanishing. 

For this, we will introduce some components as follows

\begin{align*}
&\lefteqn{\mu_{0}=\lim_{R\to\infty}\lim_{n\to\infty}\left(\int_{0}^{R}\vert u_{n}(x)\vert^{p}d\lambda_{\theta}(x)+\int_{0}^{R}\vert(u_{n})'(x)\vert^{p}d\lambda_{\alpha}(x)\right)}\\
&\mu_{\infty}=\lim_{R\to\infty}\lim_{n\to\infty}\left(\int_{R}^{\infty}\vert u_{n}(x)\vert^{p}d\lambda_{\theta}(x)+\int_{R}^{\infty}\vert (u_{n})'(x)\vert^{p}d\lambda_{\alpha}(x)\right)\\
&\nu_{0}=\lim_{R\to\infty}\lim_{n\to\infty}\left(\int_{0}^{R}A_{p,\mu}\left(\vert u_{n}\vert\right)d\lambda_{\theta}(x)\right)\\
&\nu_{\infty}=\lim_{R\to\infty}\lim_{n\to\infty}\left(\int_{R}^{\infty}A_{p,\mu}\left(\vert u_{n}(x)\vert\right)d\lambda_{\theta}(x)\right)\\
&\eta_{0}=\lim_{R\to\infty}\lim_{n\to\infty}\int_{0}^{R}\vert u_{n}(x)\vert^{\frac{p}{p-1}\lfloor p\rfloor}d\lambda_{\theta}(x)\\
&\eta_{\infty}=\lim_{R\to\infty}\lim_{n\to\infty}\int_{R}^{\infty}\vert u_{n}(x)\vert^{\frac{p}{p-1}\lfloor p\rfloor} d\lambda_{\theta}(x)
\end{align*}
taking an appropriate subsequences if necessary. It is easy to see that
\begin{eqnarray}\label{2.0.1}
\nu_{i}&\geq&\frac{\mu^{\lfloor p\rfloor}}{\lfloor p\rfloor!}\eta_{i},\, 1=\mu_{0}+\mu_{\infty},\,  d(p,\theta,\mu)=\nu_{0}+\nu_{\infty} \, \ \text{and}\\
1&\geq&\eta_{0}+\eta_{\infty}\,  (\, \text{if}\, p \,\text{is an integer}),\nonumber
\end{eqnarray}
where $i=0 \,\text{or}\, \ i=\infty$.
\begin{definition} $(u_{n})$ is a normalized vanishing sequence, $($$NVS$$)$ in short, if $(u_{n})$ satisfies  $\Vert u_{n}\Vert_{W^{1,p}_{\alpha,\theta}(0,\infty)}=1$ $($with $\alpha-(p-1)=0$$)$, $u=0$
 and $\nu_{0}=0$.
\end{definition}

\begin{example}\label{1.12}
	Let $\phi$ be a smooth nonincreasing function with compact support on $[0,+\infty)$ satisfying $\left\Vert \phi\Vert\right._{L_{\theta}^{p}(0,\infty)}=1$. Besides that, we take $\gamma,\sigma$ 
	positive real numbers such that $p\gamma-\sigma(1+\theta)=0$. We set
\begin{eqnarray*}\label{1.12.1}
	\phi_{n}(x):=\frac{\lambda_{n}^{\gamma}\phi(\lambda_{n}^{\sigma}x)}{(1+\lambda_{n}^{p\gamma}\lambda_{0})^{\frac{1}{p}}},
\end{eqnarray*}
where $\lambda_{0}:=\Vert \phi'\Vert_{L_{\alpha}^{p}(0,\infty)}^{p}$ and $(\lambda_{n})$ is a positive sequence such that $\lambda_{n}\to 0$ as $n\to\infty$. Thus, $\phi_{n}$ is a normalized vanishing
 sequence.
\end{example}
The main aim here it is show that $d(\alpha,\theta,\mu)$ is gratter than the vanishing level, more precisely

\begin{eqnarray*}
d(\alpha,\theta,\mu)>\sup_{\left\{(u_{n})\subset W^{1,p}_{\alpha,\theta}(0,\infty): (u_{n})\, \, \mbox{is a NVS}\right\}}\int_{0}^{\infty}A_{p,\mu}\left(\vert u_{n}(x)\vert\right)d\lambda_{\theta}(x).
\end{eqnarray*}

Thus, we define the \textit{normalized vanishing limit} as follows

\begin{definition}
	 The number
	\begin{eqnarray}\label{1.12.0}
	d_{nvl}(\alpha,\theta,\mu)=\sup_{\left\{(u_{n})\subset W^{1,p}_{\alpha,\theta}(0,\infty): (u_{n})\, \, \mbox{is a NVS}\right\}}\int_{0}^{\infty}A_{p,\mu}\left(\vert u_{n}(x)\vert\right)d\lambda_{\theta}(x),
	\end{eqnarray}
 is called a \textit{normalized vanishing limit}.
\end{definition}
The \textit{normalized vanishing limit} will depend only on $\alpha$ and $\mu$.

Next, we rewrite the elements defined above. Given a real number $R>0$,  we take a function $\phi_{R}\in C^{\infty}(\mathbb{R})$ such that
\begin{eqnarray*}
\left\{\begin{array}{cc} \phi_{R}(x)=1 & 0 \leq x < R \\  0\leq\phi_{R}(x)\leq 1, & R\leq x\leq R+1 \\
\phi_{R}(x)=0 & R+1\leq x \\
|\phi_{R}'(x)|\leq 2 & x\in\mathbb{R}.
\end{array}\right.
\end{eqnarray*}
After that, we define the functions $\phi_{R}^{0}$ and $\phi_{R}^{\infty}$ by
\begin{eqnarray*}
\phi_{R}^{0}(x):=\phi_{R}(x), \, \ \phi_{R}^{\infty}(x):=1-\phi_{R}^{0}(x).
\end{eqnarray*}

\begin{lemma}\label{1.11} Let $u_{n,R}^{i}=\phi_{R}^{i}u_{n}\,\ (i=0,\infty)$. We have
	\begin{eqnarray}
	&\mu_{i}&=\lim_{R\to\infty}\lim_{n\to\infty}\left(\int_{0}^{\infty}\vert u_{n,R}^{i}(x)\vert^{p}d\lambda_{\theta}(x)+\int_{0}^{\infty}\vert (u_{n,R}^{i}(x))'\vert^{p}d\lambda_{\alpha}(x)\right)\label{1.11.0}\\
&\nu_{i}&=\lim_{R\to\infty}\lim_{n\to\infty}\int_{0}^{\infty}A_{p,\mu}\left(\vert u_{n,R}^{i}(x)\vert\right)d\lambda_{\theta}(x)\label{1.11.1}\\
&\eta_{i}&=\lim_{R\to\infty}\lim_{n\to\infty}\int_{0}^{\infty}\vert u_{n,R}^{i}\vert^{\frac{p}{p-1}\lfloor p\rfloor}d\lambda_{\theta}(x)\label{1.11.2}
	\end{eqnarray} 
\end{lemma}

\begin{proof}
we will prove only (\ref{1.11.0}) with $i=0$.
On the one hand,
\begin{eqnarray}\label{1.11.3}
\int_{0}^{R}\vert u_{n}\vert^{p}d\lambda_{\theta}\leq\int_{0}^{\infty}\vert\phi_{R}^{0}u_{n}\vert^{p}d\lambda_{\theta}\leq\int_{0}^{R+1}\vert u_{n}\vert^{p}d\lambda_{\theta}.
\end{eqnarray}
On the other hand, from the Mean Value Theorem we obtain
\begin{eqnarray}\label{1.11.4}
\int_{0}^{\infty}\vert(u_{n,R}^{0})'\vert^{p}d\lambda_{\alpha}&=&\int_{0}^{\infty}\vert\phi_{R}^{0}u_{n}'+(\phi_{R}^{0})'u_{n}\vert^{p}d\lambda_{\alpha}\nonumber\\
&=&\int_{0}^{\infty}\vert\phi_{R}^{0}u_{n}'\vert^{p}d\lambda_{\alpha} + \rho_{n,R},
\end{eqnarray} where 
\begin{align*}
&\rho_{n,R}=p\int_{0}^{\infty}\vert\phi_{R}^{0}u_{n}'+t_{n}(x)(\phi_{R}^{0})'u_{n}\vert^{p-2}\phi_{R}^{0}u_{n}'(\phi_{R}^{0})'u_{n}d\lambda_{\alpha}(x)\\
\lefteqn{}&+p\int_{0}^{\infty}\vert\phi_{R}^{0}u_{n}'+t_{n}(x)(\phi_{R}^{0})'u_{n}\vert^{p-2}t_{n}(x)(\phi_{R}^{0})'u_{n}\cdot(\phi_{R}^{0})'u_{n}d\lambda_{\alpha}(x)
\end{align*} and $0\leq t_{n}(x)\leq1$.

we get \begin{eqnarray*}
\vert\rho_{n,R}\vert&\leq& p\left[\int_{0}^{\infty}\vert\phi_{R}^{0}u_{n}'+t_{n}(x)(\phi_{R}^{0})'u_{n}\vert^{p}d\lambda_{\alpha}\right]^{\frac{p-1}{p}}\left[\int_{0}^{\infty}\vert(\phi_{R}^{0})'u_{n}\vert^{p}d
\lambda_{\alpha}\right]^{\frac{1}{p}}\\
&\leq& 2p [\Vert u_{n}'\Vert_{L_{\alpha}^{p}}+2\Vert u_{n}\Vert_{L_{\alpha}^{p}(R,R+1)}]^{p-1}\Vert u_{n}\Vert_{L_{\alpha}^{p}(R,R+1)}\\
&\leq&2p\left[1+2\Vert u_{n}\Vert_{L_{\alpha}^{p}(R,R+1)}\right]^{p-1}\Vert u_{n}\Vert_{L_{\alpha}^{p}(R,R+1)}.
\end{eqnarray*}
From compactness embedding, we have $\displaystyle\lim_{n\to\infty}\Vert u_{n}\Vert_{L_{\alpha}^{p}(R,R+1)}=\Vert u\Vert_{L_{\alpha}^{p}(R,R+1)}$. Thus,
\begin{eqnarray}\label{1.11.5}
\lim_{R\to\infty}\lim_{n\to\infty}\rho_{n,R}=0.
\end{eqnarray}
We conclude  (\ref{1.11.0}) (with $i=0$)  from (\ref{1.11.3}), (\ref{1.11.4}) and (\ref{1.11.5}).
The others cases follow from similar arguments. 
\end{proof}

Next, our goal is determining the normalized vanishing limit defined at (\ref{1.12.0}).

\begin{proposition}\label{1.13}It holds that	
\begin{eqnarray*}\label{1.13.1}
 d_{nvl}(p,\theta,\mu)=\left\{\begin{array}{cc}
 \displaystyle \frac{\mu^{p-1}}{(p-1)!},&\mbox{if}\,\ p \,\ \mbox{is integer} \\ 0, & \mbox{otherwise}. \\
\end{array}\right.
\end{eqnarray*}
\end{proposition}

\begin{proof}
Again, we recall that we can suppose that $(u_{n})$ is nonincreasing, then by Lemma \ref{1.7}
\begin{eqnarray*}
|u_{n}(x)|\leq\left(\frac{1+\theta}{\omega_{\theta}}\right)^{\frac{1}{p}}\cdot\frac{1}{x^{\frac{1+\theta}{p}}}\left(\int_{0}^{\infty}|u_{n}(y)|^{p}d\lambda_{\theta}(y)\right).
\end{eqnarray*}	

Assume that $1\leq R<\infty$, then
\begin{eqnarray*}
\sum_{j=\lfloor p\rfloor+1}^{\infty}\frac{\mu^{j}}{j!}\int_{R}^{\infty}|u_{n}|^{\frac{p}{p-1}j}d\lambda_{\theta}&\leq&\sum_{j=\lfloor p\rfloor+1}^{\infty}\frac{\mu^{j}}{j!}\left(\frac{1+\theta}{\omega_{\theta}}
\right)^{\frac{j}{p-1}}\omega_{\theta}\int_{R}^{\infty}x^{\theta-\frac{(1+\theta)}{p-1}j}dx\\
&\leq &\sum_{j=\lfloor p\rfloor+1}^{\infty}\frac{\mu^{j}}{j!}\left(\frac{1+\theta}{\omega_{\theta}}\right)^{\frac{j}{p-1}}\frac{\omega_{\theta}(p-1)}{R^{(1+\theta)\left(\frac{j}{p-1}-1\right)}}\\
&\leq&\frac{\omega_{\theta}(p-1)}{R^{(1+\theta)\left(\frac{\lfloor p\rfloor+1}{p-1}-1\right)}}\sum_{j=\lfloor p\rfloor+1}^{\infty}\frac{\mu^{j}}{j!}\left(\frac{1+\theta}{\omega_{\theta}}\right)^{\frac{j}{p-1}}.
\end{eqnarray*}
Thus
\begin{eqnarray}\label{1.13.2}
\lim_{R\to\infty}\lim_{n\to\infty}\sum_{j=\lfloor p\rfloor+1}^{\infty}\frac{\mu^{j}}{j!}\int_{R}^{\infty}\vert u_{n}\vert^{\frac{p}{p-1}j}d\lambda_{\theta}=0.
\end{eqnarray}	
if $p$ is not integer, we get
\begin{eqnarray}\label{1.13.3}
\int_{R}^{\infty}\vert u_{n}\vert^{\frac{p}{p-1}\lfloor p\rfloor}d\lambda_{\theta}\leq\left(\frac{1+\theta}{\omega_{\theta}}\right)^{\frac{\lfloor p\rfloor}{p-1}}\frac{\omega_{\theta}(p-1)}{\left(\lfloor p\rfloor-
(p-1)\right)R^{(1+\theta)\left(\frac{\lfloor p\rfloor}{p-1}-1\right)}}.
\end{eqnarray}
Hence, using (\ref{1.13.2}) and (\ref{1.13.3}), we obtain $\nu_{\infty}=0$, if $p$ is not integer.

Now, if $p$ is integer, then $\lfloor p\rfloor=p-1$ and passing to subsequence if necessary, we have

\begin{eqnarray}\label{1.13.4}
\nu_{\infty}&=&\lim_{R\to\infty}\lim_{n\to\infty}\sum_{j=\lfloor p\rfloor+1}^{\infty}\frac{\mu^{j}}{j!}\int_{R}^{\infty}\vert u_{n}\vert^{\frac{p}{p-1}j}d\lambda_{\theta}\nonumber\\
&+&\lim_{R\to\infty}\lim_{n\to\infty}\frac{\mu^{p-1}}{(p-1)!}\left\Vert u_{n}\Vert\right._{L_{\theta}^{p}(R,\infty)}^{p}\nonumber\\
&=&\frac{\mu^{p-1}}{(p-1)!}\lim_{R\to\infty}\lim_{n\to\infty}\left\Vert u_{n}\Vert\right._{L_{\theta}^{p}(R,\infty)}^{p}\nonumber\\
&\leq& \frac{\mu^{p-1}}{(p-1)!}.
\end{eqnarray}

Taking $u_{n}:=\phi_{n}$ as in the Example \ref{1.12} we obtain (\ref{1.13.2}) as well. Besides, we get
\begin{eqnarray}\label{1.13.5}
\lim_{R\to\infty}\lim_{n\to\infty}\left\Vert u_{n}\Vert\right._{L_{\theta}^{p}(R,+\infty)}^{p}=\lim_{R\to\infty}\lim_{n\to\infty}\frac{\left\Vert \phi\Vert\right._{L_{\theta}^{p}(\lambda_{n}^{\sigma}R,\infty)}^{p}}{(1+
\lambda_{n}^{p\gamma}\lambda_{0})}=1.
\end{eqnarray} 

From (\ref{1.13.2}), (\ref{1.13.3}), (\ref{1.13.4}) and (\ref{1.13.5}) the proposition follows.

\end{proof}

\begin{proposition}\label{1.14} Let $p\geq 2$ be an integer number. Then 
	\begin{eqnarray*} 
	d(p,\theta,\mu)>
	\left\{\begin{array}{cl}
	\displaystyle\frac{\mu^{p-1}}{(p-1)!},&\mbox{if}\,\ p>2 \, \ \mbox{and} \, \ \mu\in(0,\mu_{\alpha,\theta}] \\
	&\\
	\displaystyle \frac{\mu^{p-1}}{(p-1)!}, & \mbox{if}\,\ p=2 \,\ \mbox{and}\,\ \mu\in\left(\frac{2}{B(2,\theta)},\mu_{\alpha,\theta}\right]. \\
	\end{array}\right.
	\end{eqnarray*}
\end{proposition}
\begin{proof}
	Let $\gamma,\sigma$ be positive real numbers such that $\gamma p-\sigma(1+\theta)=0$ and let $v\in W^{1,p}_{\alpha,\theta}(0,\infty)$. We set 
	\begin{eqnarray*}
	v_{t}(x)=t^{\gamma}v(t^{\sigma}x),\, \ \mbox{for all}\, \ t, x\in (0,\infty).
	\end{eqnarray*}
	We get
	\begin{eqnarray*}
	\lefteqn{\int_{0}^{\infty}A_{p,\mu}\left(\frac{\vert v_{t}\vert}{\left\Vert v_{t}\right\Vert_{W^{1,p}_{\alpha,\theta}(0,\infty)}}\right)d\lambda_{\theta}}\\
	&\geq&\frac{\mu^{p-1}}{(p-1)!}\left[\frac{\left\Vert v\right\Vert_{L_{\theta}^{p}}^{p}}{\left\Vert v\right\Vert_{L_{\theta}^{p}}^{p}+t^{\gamma p}\left\Vert v'\right\Vert_{L_{\alpha}^{p}}^{p}}+\frac{\mu}{p}
	\frac{t^{\frac{p\gamma}{p-1}}\left\Vert v\right\Vert_{L_{\theta}^{\frac{p}{p-1}p}}^{\frac{p}{p-1}p}}{\left(\left\Vert v\right\Vert_{L_{\theta}^{p}}^{p}+t^{\gamma p}\left\Vert v'\right\Vert_{L_{\alpha}^{p}}^{p}
	\right)^{\frac{p}{p-1}}}\right]\\
	&:=&\frac{\mu^{p-1}}{(p-1)!}h_{p,\theta,\mu}(t).
	\end{eqnarray*}

Note that $\displaystyle\lim_{t\to 0}h_{p,\theta,\mu}(t)=1$. Thus, it is sufficient to show that $h_{p,\theta,\mu}'(t)>0$ for $0<t\ll 1$. 

Through straightforward calculation we obtain

\begin{eqnarray*}
h_{p,\theta,\mu}'(t)&=&\frac{p\gamma t^{\frac{\gamma p}{p-1}-1}}{\left(\left\Vert v\right\Vert_{L_{\theta}^{p}}^{p}+t^{\gamma p}\left\Vert v'\right\Vert_{L_{\alpha}^{p}}^{p}\right)^{2}}\nonumber\\
&\cdot&\left[\frac{\mu}{p(p-1)}\left\Vert v\right\Vert_{L_{\theta}^{\frac{p}{p-1}p}}^{\frac{p}{p-1}p}\right.\left(\left\Vert v\right\Vert_{L_{\theta}^{p}}^{p}+t^{\gamma p}\left\Vert v'\right\Vert_{L_{\alpha}^{p}}^{p}
\right)^{\frac{p-2}{p-1}}\nonumber\\
&-&\frac{p}{p-1}t^{\gamma p}\left\Vert v\right\Vert_{L_{\theta}^{\frac{p}{p-1}p}}^{\frac{p}{p-1}p}\left\Vert v'\right\Vert_{L_{\alpha}^{p}}^{p}\left(\left\Vert v\right\Vert_{L_{\theta}^{p}}^{p}+t^{\gamma p}\left\Vert
 v'\right\Vert_{L_{\alpha}^{p}}^{p}\right)^{-\frac{1}{p-1}}\nonumber\\
&-&\left.\left\Vert v\right\Vert_{L_{\theta}^{p}}^{p}\left\Vert v'\right\Vert_{L_{\alpha}^{p}}^{p}t^{\gamma p-\frac{\gamma p}{p-1}}\right]
\end{eqnarray*}
Thus we get $h_{p,\theta,\mu}'(t)>0$ for $0<t\ll 1$ if $p>2$. Now, for $p=2$ is a little bit different, because

\begin{eqnarray*}
h_{2,\theta,\mu}'(t)&=&\frac{2\gamma t^{2\gamma-1}}{\left(\left\Vert v\right\Vert_{L_{\theta}^{2}}^{2}+t^{2\gamma}\left\Vert v'\right\Vert_{L_{1}^{2}}^{2}\right)^{2}}\nonumber\\
&\cdot&\left[\frac{\mu}{2}\left\Vert v\right\Vert_{L_{\theta}^{4}}^{4}-\frac{2t^{2\gamma}\left\Vert v\right\Vert_{L_{\theta}^{4}}^{4}\left\Vert v'\right\Vert_{L_{1}^{2}}^{2}}{\left(\left\Vert v\right\Vert_{L_{\theta}
^{2}}^{2}+t^{2\gamma}\left\Vert v'\right\Vert_{L_{1}^{2}}^{2}\right)}-\left\Vert v\right\Vert_{L_{\theta}^{2}}^{2}\left\Vert v'\right\Vert_{L_{1}^{2}}^{2}\right].
\end{eqnarray*}
Taking $v\in W^{1,p}_{\alpha,\theta}(0,\infty)$ such that $B(2,\theta)^{-1}=B(v)^{-1}$, we obtain $h_{2,\theta,\mu}'(t)>0$ for $0<t\ll 1$, if $\frac{2}{B(2,\theta)}<\mu\leq 2\pi(1+\theta)$, [see Proposition 
\ref{1.1.3}].
\end{proof}

\begin{lemma}\label{1.16}
	Let $u_{i}<1 \, \ (i=0,\infty)$ and let $p\geq 2$ be an integer. Then we obtain
	
	\begin{eqnarray*}\label{1.16.0}
	\lefteqn{d(p,\theta,\mu)\left\Vert u_{n,R}^{i}\right\Vert_{W^{1,p}_{\alpha,\theta}(0,\infty)}^{p}}\nonumber\\
	&\geq&\int_{0}^{\infty}A_{p,\mu}\left(\left\vert u_{n,R}^{i}\right\vert\right)d\lambda_{\theta}+\left[\left(\frac{1}{\left\Vert u_{n,R}^{i}\right\Vert_{W^{1,p}_{\alpha,\theta}(0,\infty)}^{\frac{p}{p-1}}}-1\right)
	\right.\nonumber\\
	&\cdot&\left.\int_{0}^{\infty}A_{p,\mu}\left(\left\vert u_{n,R}^{i}\right\vert\right)-\frac{\mu^{p-1}}{(p-1)!}\left\vert u_{n,R}^{i}\right\vert^{p}d\lambda_{\theta}\right].
	\end{eqnarray*}
	whenever $n$ and $R$ are sufficiently large.
	\end{lemma}

\begin{proof}
	By definition, we have
	\begin{eqnarray}\label{1.16.1}
	\lefteqn{d(\alpha,\theta,\mu)\geq\sum_{j=p-1}^{\infty}\frac{\mu^{j}}{j!}\frac{\Vert u_{n,R}^{i}\Vert_{L_{\theta}^{\frac{p}{p-1}j}}^{\frac{p}{p-1}j}}{\Vert u_{n,R}^{i}\Vert_{W^{1,p}_{\alpha,\theta}(0,\infty)}
	^{\frac{p}{p-1}j}}}\nonumber\\
	&\geq&\frac{1}{\Vert u_{n,R}^{i}\Vert_{W^{1,p}_{\alpha,\theta}(0,\infty)}^{p}}\sum_{j=p-1}^{\infty}\frac{\mu^{j}}{j!}\Vert u_{n,R}^{i}\Vert_{L_{\theta}^{\frac{jp}{p-1}}}^{\frac{jp}{p-1}}\nonumber\\
	&+&\frac{1}{\Vert u_{n,R}^{i}\Vert_{W^{1,p}_{\alpha,\theta}(0,\infty)}^{p}}\sum_{j=p}^{\infty}\left(\frac{1}{\Vert u_{n,R}^{i}\Vert_{W^{1,p}_{\alpha,\theta}(0,\infty)}^{\frac{p}{p-1}(j-(p-1))}}-1\right)
	\frac{\mu^{j}}{j!}\Vert u_{n,R}^{i}\Vert_{L_{\theta}^{\frac{jp}{p-1}}}^{\frac{jp}{p-1}}
	\end{eqnarray}
	From $\mu_{i}<1$ and (\ref{1.16.1}) we obtain
	\begin{eqnarray*}
	\lefteqn{d(\alpha,\theta,\mu)\Vert u_{n,R}^{i}\Vert_{W^{1,p}_{\alpha,\theta}(0,\infty)}^{p}}\nonumber\\
	&\geq&\sum_{j=p-1}^{\infty}\frac{\mu^{j}}{j!}\Vert u_{n,R}^{i}\Vert_{L_{\theta}^{\frac{p}{p-1}j}}^{\frac{p}{p-1}j}+\sum_{j=p}^{\infty}\left(\frac{1}{\Vert u_{n,R}^{i}\Vert_{W^{1,p}_{\alpha,\theta}(0,\infty)}
	^{\frac{p}{p-1}}}-1\right)\frac{\mu^{j}}{j!}\Vert u_{n,R}^{i}\Vert_{L_{\theta}^{\frac{p}{p-1}j}}^{\frac{p}{p-1}j}\nonumber\\
	&=&\int_{0}^{\infty}A_{p,\mu}(\vert u_{n,R}^{i}\vert)d\lambda_{\theta}\nonumber\\
	&+&\left(\frac{1}{\Vert u_{n,R}^{i}\Vert_{W^{1,p}_{\alpha,\theta}(0,\infty)}^{\frac{p}{p-1}}}-1\right)\int_{0}^{\infty}A_{p,\mu}(\vert u_{n,R}^{i}\vert)-\frac{\mu^{p-1}}{(p-1)!}\vert u_{n,R}^{i}\vert^{p}d
	\lambda_{\theta}
	\end{eqnarray*}
	for large $R$ and large $n$.
\end{proof} 

\begin{proposition}\label{1.17} Let $p\geq2$ be an integer. Then
	\begin{eqnarray*}
	(\mu_{0},\nu_{0})=(1,d(p,\theta,\mu))\, \ \mbox{and}\, \ (\mu_{\infty},\nu_{\infty})=(0,0).
	\end{eqnarray*}
\end{proposition}

\begin{proof}
	By contradiction, supposse that $0<\mu_{0}<1$. Then $0<\mu_{\infty}<1$, by relation (\ref{2.0.1}).
	From Lemma \ref{1.11} and Lemma  \ref{1.16} we have
	\begin{eqnarray}\label{1.17.0}
	d(\alpha,\theta,\mu)\mu_{i}\geq\nu_{i}+\left[\frac{1}{\mu_{i}^{\frac{1}{p-1}}}-1\right]\left[\nu_{i}-\frac{\mu^{p-1}}{(p-1)!}\eta_{i}\right].
	\end{eqnarray}
	By relation (\ref{2.0.1}) and together with (\ref{1.17.0}) we get
	\begin{eqnarray*}
	d(\alpha,\theta,\mu)\mu_{i}\geq\nu_{i},\,\  \mbox{for}\, \ i=0,\infty.
	\end{eqnarray*}

Thus, \begin{eqnarray*}
d(\alpha,\theta,\mu)=d(\alpha,\theta,\mu)(\mu_{0}+\mu_{\infty})\geq \nu_{0}+\nu_{\infty}=d(\alpha,\theta,\mu)
\end{eqnarray*}
and consequently
\begin{eqnarray*}
d(\alpha,\theta,\mu)\mu_{i}=\nu_{i}.
\end{eqnarray*}
From the last relation and (\ref{1.17.0}) we obtain
\begin{eqnarray*}\label{1.17.1}
\nu_{i}\leq\frac{\mu^{p-1}}{(p-1)!}\eta_{i},
\end{eqnarray*}
whence \begin{eqnarray*}
d(\alpha,\theta,\mu)=\nu_{0}+\nu_{\infty}\leq\frac{\mu^{p-1}}{(p-1)!}(\eta_{0}+\eta_{\infty})\leq\frac{\mu^{p-1}}{(p-1)!},
\end{eqnarray*}
which contradicts the Proposition \ref{1.14}.

Now, again, by contradiction, suppose that $\mu_{0}=0$. Thus, by Lemma \ref{1.16}
\begin{eqnarray}\label{1.17.2}
\lefteqn{d(p,\theta,\mu)\left\Vert u_{n,R}^{0}\right\Vert_{W^{1,p}_{\alpha,\theta}(0,\infty)}^{p}}\nonumber\\
&\geq&\int_{0}^{\infty}A_{p,\mu}\left(\left\vert u_{n,R}^{0}\right\vert\right)d\lambda_{\theta}\nonumber\\
&+&\frac{1}{2}\int_{0}^{\infty}\left(A_{p,\mu}\left(\left\vert u_{n,R}^{0}\right\vert\right)-\frac{\mu^{p-1}}{(p-1)!}\left\vert u_{n,R}^{0}\right\vert^{p}\right)d\lambda_{\theta}.
\end{eqnarray}
for large $R$ and large $n$.

Taking the double limit in (\ref{1.17.2}), $\lim_{R\to\infty}\lim_{n\to\infty}$,  we obtain
\begin{eqnarray*}
d(\alpha,\theta,\mu)\mu_{0}\geq\nu_{0}+\frac{1}{2}\left(\nu_{0}-\frac{\mu^{p-1}}{(p-1)!}\eta_{0}\right)\geq \nu_{0},
\end{eqnarray*}
hence $\nu_{0}=0$ from relation (\ref{2.0.1}), and $\mu_{0}=0$, getting a contradiction from Proposition \ref{1.13}, relation (\ref{2.0.1}), and Proposition \ref{1.14}.
Finally, using the same arguments we can get $\nu_{\infty}=0$ whenever $\mu_{\infty}=0$. Therefore, the proposition follows.	
\end{proof}

\noindent\textbf{Proof of Theorem \ref{0.1.15}.}

First of all, we will show that
\begin{eqnarray}\label{1.19.1}
\lim_{n\to\infty}\int_{0}^{\infty}|u_{n}|^{\frac{p}{p-1}\lfloor p\rfloor} d\lambda_{\theta}=\int_{0}^{\infty}|u|^{\frac{p}{p-1}\lfloor p\rfloor}d\lambda_{\theta}.
\end{eqnarray}

Indeed, given $R>0$, note that
\begin{eqnarray*}
\left|\int_{0}^{\infty}\left(|u_{n}|^{\frac{p}{p-1}\lfloor p\rfloor}-|u|^{\frac{p}{p-1}\lfloor p\rfloor}\right) d\lambda_{\theta}\right|&\leq& \left|\int_{0}^{R}\left(|u_{n}|^{\frac{p}{p-1}\lfloor p\rfloor}-|u|^{\frac{p}
{p-1}\lfloor p\rfloor}\right) d\lambda_{\theta}\right|\\
&+&\int_{R}^{\infty}|u_{n}|^{\frac{p}{p-1}\lfloor p\rfloor} d\lambda_{\theta}+\int_{R}^{\infty}|u|^{\frac{p}{p-1}\lfloor p\rfloor} dx\\
&=:&I(n, R)+II(n, R)+III(R).
\end{eqnarray*}
By compact embedding we have $\lim_{R\to\infty}\lim_{n\to\infty}I(n,R)=0$. From Dominated Convergence Theorem, we obtain $\lim_{R\to\infty}\lim_{n\to\infty}III(R)=0$. If $p$ is an integer, we get 
$\lim_{R\to\infty}\lim_{n\to\infty}II(n, R)=0$ from $\mu_{\infty}=0$ (Proposition \ref{1.17}). If $p\notin\mathbb{N}$, we obtain $\lim_{R\to\infty}\lim_{n\to\infty}II(n, R)=0$ from inequality (\ref{1.13.3}). 
 Hence, (\ref{1.19.1}) follows.
Now, assume that either $p>2$ and $\mu\in(0,\mu_{\alpha,\theta})$ or $p=2$ and $\alpha\in(2/B(2,\theta),\mu_{1,\theta})$. Writing
\begin{eqnarray*}\label{1.19.2}
d(p,\theta,\mu)-\int_{0}^{\infty} A_{p,\mu}\left(\left\vert u\right\vert\right)d\lambda_{\theta}&=&\int_{0}^{\infty}A_{p,\mu}\left(\left\vert u_{n}\right\vert\right)d\lambda_{\theta}-\int_{0}^{\infty}A_{p,\mu}\left(\left
\vert u\right\vert\right)d\lambda_{\theta}\nonumber\\
&+&\left(d(p,\theta,\mu)-\int_{0}^{\infty}A_{p,\mu}\left(\left\vert u_{n}\right\vert\right)d\lambda_{\theta} \right)\nonumber\\
&=:&IV(n)+V(n),
\end{eqnarray*}
where,
\begin{eqnarray*}
V(n):=d(p,\theta,\mu)-\int_{0}^{\infty}A_{p,\mu}\left(\left \vert u_{n}\right\vert\right)d\lambda_{\theta}
\end{eqnarray*}
and
\begin{eqnarray*}
IV(n):=\int_{0}^{\infty}A_{p,\mu}\left(\left\vert u_{n}\right\vert\right)d\lambda_{\theta}-\int_{0}^{\infty}A_{p,\mu}\left(\left\vert u\right\vert\right)d\lambda_{\theta}.
\end{eqnarray*}
We get, by definition of $d(\alpha,\theta,\mu)$, that
\begin{eqnarray*}
\lim_{R\to\infty}\lim_{n\to\infty} V(n)=0.
\end{eqnarray*}

Since
\begin{align*}
IV(n)&=\int_{0}^{\infty}\left(A_{p,\mu}\left(\left\vert u_{n}\right\vert\right)-\frac{\mu^{\lfloor p\rfloor}}{\lfloor p\rfloor!}\left\vert u_{n}\right\vert^{\frac{p}{p-1}\lfloor p\rfloor}\right) d\lambda_{\theta}\nonumber\\
&-\int_{0}^{\infty}\left(A_{p,\mu}\left(\left\vert u\right\vert\right)-\frac{\mu^{\lfloor p\rfloor}}{\lfloor p\rfloor!}\left\vert u\right\vert^{\frac{p}{p-1}\lfloor p\rfloor}\right) d\lambda_{\theta}\\
&+\frac{\mu^{\lfloor p\rfloor}}{\lfloor p\rfloor!}\int_{0}^{\infty}\left(\left\vert u_{n}\right\vert^{\frac{p}{p-1}\lfloor p\rfloor}-\left\vert u\right\vert^{\frac{p}{p-1}\lfloor p\rfloor}\right) d\lambda_{\theta}
\end{align*}
From Lemma \ref{2.1.0} and relation (\ref{1.19.1}) we obtain
\begin{align*}
\lim_{n\to\infty} IV(n)=0.
\end{align*}

Now, we assert that $\| u\|_{W^{1,p}_{p-1,\theta}(0,\infty)}=1$. Indeed, on the one hand, 
\[\left\Vert u\right\Vert_{W^{1,p}_{p-1,\theta}(0,\infty)}\leq\liminf_{n\to\infty}\left\Vert u_{n}\right\Vert_{W^{1,p}_{p-1,\theta}(0,\infty)}= 1.\]
 On the other hand,
\begin{eqnarray*}
d(p,\theta,\mu)&\geq&\int_{0}^{\infty}A_{p,\mu}\left(\left(\frac{\left\vert u\right\vert}{\left\Vert u\right\Vert_{W^{1,p}_{p-1,\theta}(0,\infty)}}\right)\right)d\lambda_{\theta}\\
&=&\sum_{j=\lfloor p\rfloor}^{\infty}\frac{\mu^{j}}{j!}\frac{\left\Vert u\right\Vert_{L_{\theta}^{\frac{p}{p-1}j}}^{\frac{p}{p-1}j}}{\left\Vert u\right\Vert_{W^{1,p}_{p-1,\theta}(0,\infty)}^{\frac{p}{p-1}j}}\nonumber\\
&\geq&\frac{1}{\left\Vert u\right\Vert_{W^{1,p}_{p-1,\theta}(0,\infty)}^{\frac{p}{p-1}\lfloor p\rfloor}}\sum_{j=\lfloor p\rfloor}^{\infty}\frac{\mu^{j}}{j!}\left\Vert u\right\Vert_{L_{\theta}^{\frac{p}{p-1}j}}^{\frac{p}{p-1}
j}\\
&\geq&\frac{1}{\left\Vert u\right\Vert_{W^{1,p}_{p-1,\theta}(0,\infty)}^{\frac{p}{p-1}\lfloor p\rfloor}}\cdot d(p,\theta,\mu).
\end{eqnarray*}
Therefore, $\left\Vert u\right\Vert_{W^{1,p}_{p-1,\theta}(0,\infty)}=1$ and \begin{eqnarray*}
d(p,\theta,\mu)=\int_{0}^{\infty}A_{p,\mu}\left(\left\vert u\right\vert\right)d\lambda_{\theta}.
\end{eqnarray*}


\section{Proof of the Theorem \ref{0.1.15}}
\label{sec:6}
Throughout this section, we assume that $p=2$ and $\mu\leq\pi(1+\theta)/3$. 

By Theorem \ref{P.0.1.7} (inequality (\ref{0.1.7.0})), we get
\begin{eqnarray}\label{4.0.1}
\frac{\Vert u\Vert_{L_{\theta}^{2j}}^{2j}}{\Vert u'\Vert_{L_{1}^{2}}^{2}\cdot\Vert u\Vert_{L_{\theta}^{2}}^{2}}\leq C_{\gamma,2,\theta}\frac{j!}{\gamma^{j}}\Vert u'\Vert_{L_{1}^{2}}^{2(j-2)}
\end{eqnarray}
for all $u\in W^{1,2}_{1,\theta}(0,\infty)$, $j\in\mathrm{N}$, and $0<\gamma<(1+\theta)\omega_{1}$. We are going to use the inequality (\ref{4.0.1}) to prove the Theorem \ref{0.1.15}.

\noindent\textbf{Proof of Theorem \ref{0.1.15}}

Let $S:=\{v\in W^{1,2}_{1,\theta}(0,\infty): \Vert v\Vert_{W^{1,2}_{1,\theta}(0,\infty)}=1\}$. For each $u\in S$, we define a family of functions by
\begin{eqnarray*}\label{4.0.2}
u_{t}(x):=t^{\frac{1}{2}}u(t^{\frac{1}{1+\theta}}x),
\end{eqnarray*}	
where $t>0$ is a parameter. Besides, let $v_{t}:= u_{t}/\Vert u_{t}\Vert_{W^{1,2}_{1,\theta}(0,\infty)}$. Thus $v_{t}$ is a curve in $S$ passing through $u$ when $t=1$. Then it is sufficient to show that

\begin{eqnarray*}\label{4.0.3}
\frac{d}{dt}F(v_{t})\left\|_{t=1}<0,\right.
\end{eqnarray*}
where $F(w):=\displaystyle\int_{0}^{\infty} A_{2,\mu}\left(w(x)\right)d\lambda_{\theta}(x)$.

Through a direct calculation we have that $\Vert u_{t}\Vert_{L_{\theta}^{2j}}^{2j}=t^{j-1}\Vert u\Vert_{L_{\theta}^{2j}}^{2j}$, $\Vert (u_{t})'\Vert_{L_{1}^{2}}^{2}=t\Vert u'\Vert_{L_{1}^{2}}^{2}$ and
\begin{align*}
F(v_{t})=
\sum_{j=1}^{\infty}\frac{\mu^{j}}{j!}\frac{t^{j-1}\Vert u\Vert_{L_{\theta}^{2j}}^{2j}}{\left(\Vert u\Vert_{L_{\theta}^{2}}^{2}+t\Vert u'\Vert_{L_{1}^{2}}^{2}\right)^{j}}.
\end{align*}
Since \begin{eqnarray*}
\lefteqn{\frac{d}{dt}F(v_{t})=}\\
&\sum_{j=1}^{\infty}\frac{\mu^{j}}{j!}\frac{(j-1)t^{j-2}\Vert u\Vert_{L_{\theta}^{2j}}^{2j}\left(\Vert u\Vert_{L_{\theta}^{2}}^{2}+t\Vert u'\Vert_{L_{1}^{2}}^{2}\right)^{j}-j t^{j-1}\Vert u'\Vert_{L_{1}^{2}}^{2}\Vert u
\Vert_{L_{\theta}^{2j}}^{2j}\left(\Vert u\Vert_{L_{\theta}^{2}}^{2}+t\Vert u'\Vert_{L_{1}^{2}}^{2}\right)^{j-1}}{\left(\Vert u\Vert_{L_{\theta}^{2}}^{2}+t\Vert u'\Vert_{L_{1}^{2}}^{2}\right)^{2j}}
\end{eqnarray*}
we obtain
\begin{align}\label{4.0.4}
\frac{d}{dt}F(v_{t})\left\|_{t=1}\right.&=\sum_{j=1}^{\infty}\frac{\mu^{j}}{j!}\Vert u\Vert_{L_{\theta}^{2j}}^{2j}\left[(j-1)\Vert u\Vert_{L_{\theta}^{2}}^{2}-\Vert u'\Vert_{L_{1}^{2}}^{2}\right]\nonumber\\
&=-\mu\Vert u\Vert_{L_{\theta}^{2}}^{2}\Vert u'\Vert_{L_{1}^{2}}^{2}+\sum_{j=2}^{\infty}\frac{\mu^{j}}{j!}\Vert u\Vert_{L_{\theta}^{2j}}^{2j}\left[(j-1)\Vert u\Vert_{L_{\theta}^{2}}^{2}-\Vert u'\Vert_{L_{1}^{2}}
^{2}\right]\nonumber\\
&\leq\mu\Vert u\Vert_{L_{\theta}^{2}}^{2}\Vert u'\Vert_{L_{1}^{2}}^{2}\left(-1+\sum_{j=2}^{\infty}\frac{\mu^{j-1}}{(j-1)!}\frac{\Vert u\Vert_{L_{\theta}^{2j}}^{2j}}{\Vert u\Vert_{L_{\theta}^{2}}^{2}\Vert
 u'\Vert_{L_{1}^{2}}^{2}}\right).
\end{align}
From inequality (\ref{4.0.1}) (with $\gamma:=2\pi(1+\theta)/3$) and (\ref{4.0.4}) we get
\begin{align*}
&\frac{d}{dt}F(v_{t})\left|_{t=1}\right.\nonumber\\
&\leq\mu\Vert u\Vert_{L_{\theta}^{2}}^{2}\Vert u'\Vert_{L_{1}^{2}}^{2}\left(-1+C_{\frac{2}{3}\pi(1+\theta),2,\theta}\sum_{j=2}^{\infty}\frac{\mu^{j-1}}{(j-1)!}j!\left(\frac{3}{2\pi(1+\theta)}\right)^{j}\right)
\nonumber\\
&=\mu\Vert u\Vert_{L_{\theta}^{2}}^{2}\Vert u'\Vert_{L_{1}^{2}}^{2}\\
&\cdot\left(-1+C_{\frac{2}{3}\pi(1+\theta),2,\theta}\left(\frac{3}{2\pi(1+\theta)}\right)^{2}\mu\sum_{j=2}^{\infty}\mu^{j-2}j\left(\frac{3}{2\pi(1+\theta)}\right)^{j-2}\right)\\
&\leq\mu\Vert u\Vert_{L_{\theta}^{2}}^{2}\Vert u'\Vert_{L_{1}^{2}}^{2}\left(-1+C_{\frac{2}{3}\pi(1+\theta),2,\theta}\left(\frac{3}{2\pi(1+\theta)}\right)^{2}\mu\sum_{j=2}^{\infty}j\left(\frac{1}{2}\right)^{j-2}\right).
\end{align*}
Thus, taking $\mu_{0}:=\frac{1}{a\cdot C_{\frac{2}{3}\pi(1+\theta),2,\theta}}\left(\frac{2\pi(1+\theta)}{3}\right)^{2}$, where $a:=\displaystyle\sum_{j=2}^{\infty}j\left(\frac{1}{2}\right)^{j-2}$, the proof of the
 theorem follows.

\section{Gagliardo-Nirenberg Inequalities}
\label{sec:7}
In this section, we discuss a little bit about the best constant of the Gagliardo-Niremberg inequality, and we will explore some ideas contained in \cite{Agueh,Caffarelli,Wang}.

It is known the interpolation inequality with weights  \begin{eqnarray}\label{1.1.1}
\left\Vert u\right\Vert_{L_{\theta}^{q}(0,\infty)}\leq K(p,q,\alpha,\theta)\left\Vert u'\right\Vert_{L_{\alpha}^{p}(0,\infty)}^{\gamma}\left\Vert u\right\Vert_{L_{\theta}^{p}(0,\infty)}^{1-\gamma},
\end{eqnarray}

where $1<p\leq q< p^{\star}=\frac{p(1+\theta)}{\alpha-(p-1)}$, $\alpha\geq p-1$, $\theta\geq 0$ and $1-\gamma=\frac{p}{q}\cdot\frac{(p^{\star}-q)}{(p^{\star}-p)}$. It is worth noting that when $\alpha=p-1$
 we have $1-\gamma=\frac{p}{q}$.  

Throughout this section we will assume that $\alpha\leq p+\theta$. So, we can computed the optimal $k=K(p,q,\alpha,\theta)$ in  (\ref{1.1.1}) if we determine the explicit solution of the minimization 
problem 
\begin{eqnarray}\label{1.1.2}
\inf\left\{E(u):=\frac{1}{p}\int_{0}^{\infty}\left\vert u'\right\vert^{p}d\lambda_{\alpha}+\frac{1}{p}\int_{0}^{\infty}\left\vert u\right\vert^{p}d\lambda_{\theta}: \left\Vert u\right\Vert_{L_{\theta}^{q}((0,\infty))}
=1\right\}.
\end{eqnarray}

Indeed, first of all, see Lemma \ref{1.7} and Theorem \ref{2} (with $\alpha=m$, and $l=\theta$) for the existence of a minimizer for (\ref{1.1.2}). Now if $u_{\infty}$ is a minimizer of the variational problem
 (\ref{1.1.2}), then
\begin{eqnarray*}
E(u_{\infty})\leq E(u)=\frac{1}{p}\left\Vert u'\right\Vert_{L_{\alpha}^{p}((0,\infty))}^{p}+\frac{1}{p}\left\Vert u\right\Vert_{L_{\theta}^{p}((0,\infty))}^{p}
\end{eqnarray*}
for all $u\in W^{1,p}_{\alpha,\theta}((0,\infty))$ satisfying $\left\Vert u\right\Vert_{L_{\theta}^{q}((0,\infty))}=1$. Thus, \begin{eqnarray*}
E(u_{\infty})\leq\frac{1}{p}\frac{\left\Vert u'\right\Vert_{L_{\alpha}^{p}((0,\infty))}^{p}}{\left\Vert u\right\Vert_{L_{\theta}^{q}((0,\infty))}^{p}}+\frac{1}{p}\frac{\left\Vert u\right\Vert_{L_{\theta}^{p}((0,\infty))}^{p}}
{\left\Vert u\right\Vert_{L_{\theta}^{q}((0,\infty))}^{p}}
\end{eqnarray*} 
for every $0\neq u\in W^{1,p}_{\alpha,\theta}(0,\infty)$. Scaling $u$ as $u_{t}(x)=u(tx)$, we get
\begin{eqnarray*}
E(u_{\infty})\leq t^{p-(\alpha+1)+\frac{p}{q}(1+\theta)}\frac{\left\Vert u'\right\Vert_{L_{\alpha}^{p}((0,\infty))}^{p}}{p\left\Vert u\right\Vert_{L_{\theta}^{q}((0,\infty))}^{p}}+t^{(1+\theta)\left(\frac{p}{q}-1\right)}
\frac{\left\Vert u\right\Vert_{L_{\theta}^{p}((0,\infty))}^{p}}{p\left\Vert u\right\Vert_{L_{\theta}^{q}((0,\infty))}^{p}}.
\end{eqnarray*}
A direct computation proves that the minimum over $t$ is achieved at
\begin{eqnarray*}
t=\left[\frac{(1+\theta)(q-p)}{pq+p(1+\theta)-q(1+\alpha)}\frac{B}{A}\right]^{\frac{1}{p+\theta-\alpha}},
\end{eqnarray*}
where \begin{eqnarray*}
A=\frac{\left\Vert u'\right\Vert_{L_{\alpha}^{p}((0,\infty))}^{p}}{p\left\Vert u\right\Vert_{L_{\theta}^{q}((0,\infty))}^{p}} \, \ \mbox{and} \, \ B=\frac{\left\Vert u\right\Vert_{L_{\alpha}^{p}((0,\infty))}^{p}}{p\left\Vert
 u\right\Vert_{L_{\theta}^{q}((0,\infty))}^{p}}. 
\end{eqnarray*}
Therefore, \begin{eqnarray*}
E(u_{\infty})\leq&\left[\left(\frac{(1+\theta)(q-p)}{qp+p(1+\theta)-q(1+\alpha)}\right)^{1-\gamma}+\left(\frac{(1+\theta)(q-p)}{qp+p(1+\theta)-q(1+\alpha)}\right)^{\gamma}\right]\frac{\left\Vert u'\right
\Vert_{L_{\alpha}^{p}}^{p\gamma}\left\Vert u\right\Vert_{L_{\theta}^{p}}^{p(1-\gamma)}}{p\left\Vert u\right\Vert_{L_{\theta}^{q}}^{p}}
\end{eqnarray*}
and the equality happen when $u=u_{\infty}$.

The next result will be important in the study of attainability in the Trudinger-Moser inequality with weigth when $p=2$.
\begin{proposition}\label{1.1.3} 
If $p=2$, $\alpha=p-1$, and $\theta\geq0$, then the infimum
	\begin{eqnarray*}
	B(2,\theta)^{-1}:=\inf_{0\neq u\in W^{1,2}_{1,\theta}(0,\infty)}\frac{\left\Vert u'\right\Vert_{L_{1}^{2}((0,\infty))}^{2}\cdot\left\Vert u\right\Vert_{L_{\theta}^{2}((0,\infty))}^{2}}{\left\Vert u\right
	\Vert_{L_{\theta}^{4}((0,\infty))}^{4}},	
	\end{eqnarray*}
	is attained by a positive nonincreasing function  in $W^{1,2}_{1,\theta}((0,\infty))$. Moreover, 
	\[B(2,\theta)^{-1}<\pi(1+\theta).\]	
\end{proposition}
\begin{proof}
	The first part has been discuss at the beginning of this section. Then, we focus in the second part.
	Set \begin{align*} B(u)^{-1}:=\frac{\left\Vert u'\right\Vert_{L_{1}^{2}((0,\infty))}^{2}\cdot\left\Vert u\right\Vert_{L_{\theta}^{2}((0,\infty))}^{2}}{\left\Vert u\right\Vert_{L_{\theta}^{4}((0,\infty))}^{4}}.
	\end{align*}
Note that is sufficient to exhibits a function $u\in W^{1,2}_{1,\theta}((0,\infty))$ such that\linebreak $B(u)^{-1}=\pi(1+\theta)$ and it is not solution of 
\begin{eqnarray}\label{1.1.5}
-(u'x)'\omega_{1}+ux^{\theta}\omega_{\theta}-\lambda u^{3}x^{\theta}\omega_{\theta}=0,
\end{eqnarray}
for all $\lambda>0$. 
	
On the one hand, through a direct calculation we can see that for every positive function $v$ in $ W^{1,2}_{1,\theta}((0,\infty))$ of the form \begin{eqnarray*}
v(x)=a_{1}(1+a_{2}x^{a_{3}})^{a_{4}},
\end{eqnarray*}
where $a_{1}, a_{2}, a_{3}, a_{4}$ are real numbers, it is not a solution for (\ref{1.1.5}).
On the other hand, choosing \begin{eqnarray*}
u(x)=\frac{1}{1+x^{1+\theta}},
\end{eqnarray*}
	then \begin{eqnarray*}
&\left\Vert u\right\Vert_{L_{\theta}^{4}((0,\infty))}^{4}=\frac{\omega_{\theta}}{3(1+\theta)}\\
&\left\Vert u\right\Vert_{L_{\theta}^{2}((0,\infty))}^{2}=\frac{\omega_{\theta}}{(1+\theta)} \\
& \left\Vert u'\right\Vert_{L_{1}^{2}((0,\infty))}^{2}=\frac{\omega_{1}(1+\theta)}{6}
\end{eqnarray*}
	Therefore, $B(u)^{-1}=\pi(1+\theta)$, and then the proposition follows.
\end{proof}

\end{document}